\documentclass[preprint,authoryear,12pt]{elsarticle}
\usepackage[round,authoryear]{natbib}
\usepackage{times}
\usepackage{mathptmx}
\usepackage{geometry}
\usepackage{caption}
\usepackage{amssymb}
\usepackage{amsmath}
\usepackage{graphicx}
\usepackage{indentfirst}
\usepackage{amsthm}
\usepackage{subfigure}
\usepackage{amsfonts}
\usepackage{booktabs}
\usepackage{shorttoc}
\usepackage{array}
\usepackage{multirow}
\usepackage{listings}
\usepackage{xcolor}
\usepackage{color}
\usepackage{amsopn}
\usepackage{float}
\usepackage{enumerate}
\usepackage{bbding}
\usepackage{subfigure}
\usepackage{arydshln}
\makeatletter
\newcommand{\rmnum}[1]{\romannumeral #1}
\newcommand{\Rmnum}[1]{\expandafter\@slowromancap\romannumeral #1@}
\makeatother

\begin{document}
\def\@captype{figure}
\newtheorem{thm}{\emph{Theorem}}[section]
\newtheorem{defn}[thm]{\emph{Definition}}
\newtheorem{prop}[thm]{\emph{Proposition}}
\newtheorem{rem}[thm]{\emph{Remark}}
\newtheorem{exa}[thm]{\emph{Example}}
\newtheorem{lem}[thm]{\emph{Lemma}}
\newtheorem{cor}[thm]{\emph{Corollary}}
\newtheorem*{ass}{\emph{Assumption}}

\begin{frontmatter}



\title{Conditional Tail-Related Risk Estimation Using Composite Asymmetric Least Squares and Empirical Likelihood}
\author[label2]{Sheng Wu}\ead{11535034@zju.edu.cn}
\author[label2]{Yi Zhang}
\address[label2]{School of Mathematical Sciences, Zhejiang University}
\author[label3]{Jun Zhao}
\address[label3]{Zhejiang University City College}
\author[label4]{Liming Shen}
\address[label4]{Asset Management Department, Bank of Hangzhou}

\begin{abstract}
In this article, by using composite asymmetric least squares (CALS) and empirical likelihood, we propose a two-step procedure to estimate the conditional value at risk (VaR) and conditional expected shortfall (ES) for the GARCH series. First, we perform asymmetric least square regressions at several significance levels to model the volatility structure and separate it from the innovation process in the GARCH model. Note that expectile can serve as a bond to make up the gap from VaR estimation to ES estimation because there exists a bijective mapping from expectiles to specific quantile, and ES can be induced by expectile through a simple formula. Then, we introduce the empirical likelihood method to determine the relation above; this method is data-driven and distribution-free. Theoretical studies guarantee the asymptotic properties, such as consistency and the asymptotic normal distribution of the estimator obtained by our proposed method. A Monte Carlo experiment and an empirical application are conducted to evaluate the performance of the proposed method. The results indicate that our proposed estimation method is competitive with some alternative existing tail-related risk estimation methods.
\end{abstract}

\end{frontmatter}

\newpage

\section{Introduction}
The accurate assessment of the exposure to market risk lies at the core of risk control and portfolio management. Value at risk (VaR), which was first introduced in 1990s as a risk measure, has witnessed a great development and wide applications in finance-related fields \citep{Jorion2000Value} due to its conceptual simplicity and practical convenience. From the perspective of statistics, VaR actually amounts to the quantile of a random loss variable and measures the maximum potential loss at a given specific confidence level. However, although VaR is employed as the standard risk measure by Basel \textrm{II}, it is criticised because of its lack of subadditivity, especially in portfolio management, because it is generally accepted that the aggregate risk  on a portfolio should not be greater than the sum of the risks of its constituents, but VaR does not reflect this feature. Furthermore, \cite{Lucas1998Extreme} noted that VaR ignores the extreme loss beyond itself, which
may cause some uncontrollable and hazardous loss.

These shortcomings of VaR motivated the development of another risk measure, the expected shortfall (ES), which was introduced by \cite{artzner_coherent_1999}. The ES risk measure is defined as the conditional expectation
of the loss exceeding or equal to VaR at a given confidence level. ES has been studied in detail, and it has been shown that ES
possesses good properties, such as monotonicity, sub-additivity, homogeneity, and translational invariance. In other words, ES enjoys coherence; see \cite{Pflug2000Some, acerbi_expected_2001, acerbi_coherence_2002}. This distinguished property has caused ES to be increasingly widely used in finance-related fields, such as portfolio management, risk control and prediction.

Considering the respective merits of VaR and ES, such as VaR's conceptual simplicity and ES's coherence, ES and VaR have recently been employed simultaneously  to obtain a deeper and more accurate understanding of risk management, especially in the analysis of financial time series data. Estimating (or forecasting) the conditional VaR and ES of time series is a great challenge and has attracted heated discussion for a long time; see \cite{mcneil_estimation_2000, engle_caviar_2004, cai_nonparametric_2009, taylor_estimating_2008, xiao_conditional_2009, kuan_assessing_2009}, and so on.

\cite{engle_caviar_2004} provided a review of the VaR literature and divided the corresponding estimation or prediction methods
into three different categories: parametric, semiparametric, and nonparametric. The detailed summary in \cite{engle_caviar_2004} provides guidance about the general research framework of VaR and ES since ES estimation follows  similar patterns. Generally speaking, parametric methods need a specific parameterized distribution assumption regarding financial prices. One of the most commonly used parametric methods for time series data is the volatility-based method, in which VaR is estimated based on a conditional volatility forecast with a distribution assumption for the shape of residuals. GARCH models \citep{Bollerslev1986Generalized} are the most widely used models for forecasting volatility \citep{Granger2003Forecasting}, and there are different choices for the residual distribution, such as the normal distribution \citep{Bollerslev1986Generalized}, Student or skewed Student distributions \citep{zhu_modeling_2011}, generalized Pareto distribution \citep{fotios_c._harmantzis_empirical_2006}, Johnson family \citep{simonato_performance_2011}, and mixture distribution \citep{broda_expected_2011}. This type of
approach to VaR estimation has  an appealing advantage in that it provides the structure of the data generation process, so it is very convenient and has comparable accuracy for forecasting or predicting the future VaR. However, this approach focuses on estimating VaR, and it is not clear how to obtain the corresponding ES estimates because in some situations, such as portfolio management, VaR is insufficient for describing the total risk. Non-parametric methods are another choice for VaR and ES estimation. \cite{cai_regression_2002} first applied a kernel-based method to estimate VaR. On the basis of this work, many other nonparametric methods for VaR and ES estimation have been developed, such as \cite{scaillet_nonparametric_2004,chen_nonparametric_2008,cai_nonparametric_2008,cai_nonparametric_2009}.

Compared with parametric methods, kernel-based non-parametric methods do not require specification of the distribution and are thus more flexible. However, it is well known that kernel methods may lose efficiency and be impractical for financial problems because they require larger data sets to obtain a comparable estimation accuracy \citep{Fan2006Nonlinear}. Semiparametric methods are a good alternative to balance the tradeoff of estimation efficiency and distribution-free demand; see \cite{hang_chan_interval_2007,linton_estimation_2013,wang_conditional_2016}. One of the semiparametric approaches for VaR and ES estimation is based on extreme value  theory (EVT). For example, for the GARCH model, \cite{mcneil_estimation_2000} proved that the distribution of the residuals standardized by GARCH
conditional volatility estimates beyond some threshold can be approximated by some extreme value distribution and
proposed the peaks over threshold EVT method to obtain the corresponding VaR and ES estimations.

The semiparametric autoregressive model is another appealing approach to VaR estimation. \cite{engle_caviar_2004} proposed the conditional autoregressive value at risk (CAViaR) model for time series and adopted quantile regression for coefficient estimation. The CAViaR model deals directly with the quantile process instead of the whole distribution of financial returns. Consequently, it does not require a specification of the distribution of financial returns, i.e., it does not rely on distributional assumptions, which is a quite appealing advantage in practice. However, the CAViaR model may cause inconvenience when estimating characteristic features other than quantile, such as the volatility of financial data. Considering this demand, \cite{xiao_conditional_2009} applied quantile regression to the widely used financial data generation process - the GARCH model for conditional VaR estimation. However, as \cite{taylor_estimating_2008} noted, it is still unclear how to estimate the corresponding ES from the VaR estimate in the CAViaR or GARCH models. The gap between the VaR and ES estimations is made up for by expectile. \cite{aigner_estimation_1976} and \cite{newey_asymmetric_1987} adopted the `asymmetric' concept from quantile regression in a smooth manner and proposed asymmetric least squares estimation, from which expectile originates. \cite{Efron_1991} showed that there exists a bijective mapping from expectile to quantile; i.e., for each $\alpha$-quantile of some random variable, there exists a unique corresponding $\tau$-expectile equivalent to the $\alpha$-quantile. This wonderful property  makes expectile serve as a bond between VaR and ES. The pioneering framework of ES estimation from VaR using expectile was proposed by \cite{taylor_estimating_2008}. Applying the key idea to treat the quantile structure in the CAViaR model to expectile, Taylor introduced the conditional autoregressive expectile (CARE) model to estimate VaR and ES simultaneously for time series. Since then, VaR and ES estimations using expectile have seen wide discussion and development. \cite{kuan_assessing_2009} modified the CARE model and studied the asymptotic property of this method. \cite{xie_varying-coefficient_2014} generalised the CARE model to situations with time-varying coefficients. \cite{kim_nonlinear_2016} recently extended this idea to the nonlinear case.

In the use of  CARE-type models, the relationship between VaR and ES is built up by the bijective mapping between $\tau$-expectile and $\alpha$-quantile, so a fundamental problem of great concern is how to determinate the corresponding bijective mapping $\tau$ for a fixed $\alpha$. For this problem, \cite{taylor_estimating_2008} first calculated the $\tau$-expectile and $\alpha$-quantile for a sequence of  $\tau$ and $\alpha$ using historical data and determined the corresponding mapping $\tau$ for a fixed $\alpha$ via grid-search. This approach is straightforward,  but it can be shown that it may lose estimation accuracy when estimating the conditional tail-related risk when the historical financial data are insufficient (see our simulation results in Subsection 2.1). \cite{kim_nonlinear_2016} obtained this $\tau$ value, provided that the innovation process followed the normal distribution. These methods for determining the bijective mapping are either demanding for large data sets or rely on a pre-specified distribution. Another notable issue is indicated in the empirical study of \cite{kuan_assessing_2009}, which reminds us that this $\tau$ value may be time-varying. In this article, we propose using the data-driven and distribution-free empirical likelihood method introduced by \cite{owen_empirical_1990} and \cite{qin_empirical_1994} to determine the mapping  expectile level $\tau$ for a fixed $\alpha$ for the innovation term. Before we start, a pre-processing operation must be performed to separate the innovation process from the volatility part. For this purpose, we adopt the idea used by \cite{xiao_conditional_2009} and propose using the composite asymmetric least regression.

To summarise, in this article, we estimate the volatility and conditional tail-related risks of the financial return series under the GARCH framework. Under the GARCH framework, the induced dynamic autoregressive structure of conditional tail-related risks stems from the dynamic structure of volatility, which is different from the directly portrayed autoregressive risk structure, such as CAViaR in \cite{engle_caviar_2004} or CARE in \cite{taylor_estimating_2008, kuan_assessing_2009}. We assume that the innovation process is an i.i.d. random sequence, which is commonly used in the GARCH model; see \cite{berkes_garch_2003, hall_inference_2003}. This assumption does not involve any distribution assumption but indeed helps gain a complete picture of the return series. To better capture the dynamic structure of volatility, we adopt the idea of \cite{xiao_conditional_2009} and propose CALS. Compared to the method in \cite{xiao_conditional_2009}, the main improvement of our proposed method is that we divide all autoregressive parameters into two parts: parameters from the dynamic structure of volatility and parameters from the conditional distribution of the innovation term. With such a representation, we can have a more intuitive understanding of the dynamic structure of expectile, which is the fundamental concern of our methods. This division of the coefficients can help us obtain the volatility structure from CALS directly without complex matrix decomposition computations. Once the volatility part is modelled, we can separate the innovation process from the return series. Then, we can determine the corresponding bijective mapping $\tau$ for a fixed $\alpha$ such that the $\tau$-expectile equals the $\alpha$-quantile via the empirical likelihood method. Combining CALS and empirical likelihood, the conditional VaR and ES in a GARCH framework can be estimated.

The article is organised as follows. Section 2 reviews three tail-related risks and their relation, which lead to some potential issue in estimating conditional tail-related risk. Based on these issue, we state the motivations and contributions of the proposal. In section 3, we introduce our method to estimate the conditional VaR and ES by combining CALS and empirical likelihood. The asymptotic properties of the method are presented in section 4. Simulation results and an empirical application of the method are given in section 5 and section 6, respectively. Finally, section 7 concludes the paper and presents further discussion. The proofs of some theorems in section 4 are provided in the appendix.

\section{Review of related risks, potential issues and motivation}
In this section, we first review the definition and some properties of three tail-related risks (VaR, ES and expectile), from which the estimating equations in empirical likelihood are derived. Then, we analyse the dynamic structure of tail-related risks for GARCH-type return series and state reasons why we capture the dynamic structure using CALS. Finally, we indicate some fundamental issues in the overall estimating procedure and highlight the motivation and contribution of our proposed method.

\subsection{Review of three tail-related risks}
VaR and ES are two widespread risk measures in the field of risk management. Since we are dedicated to return series, here, we consider the downside risk, similarly used in \cite{acerbi_expected_2001, engle_caviar_2004, taylor_estimating_2008}. Hence, the VaR of a random variable $X$ with significance level $\alpha$ is defined as
\begin{eqnarray}
Q_{\alpha}(X)\triangleq\inf\{x|F_X(x)\geq\alpha\},
\end{eqnarray}
where $F_X$ is the cumulative distribution function of $X$. If $F_X$ is continuous, the corresponding ES is defined as
\begin{eqnarray}
ES_{\alpha}(X)=\mathbf{E}[X|X<Q_{\alpha}(X)]=\frac{1}{\alpha}\mathbf{E}[X\cdot I(X<Q_{\alpha}(X))].
\end{eqnarray}

Both risk measures have their own merits and defects (see the Introduction section). In addition, to evaluate and backtest the risk measures, elicitability is another considerable property. Briefly, the elicitability of a risk measure determines whether we can find a scoring function from which we can obtain the optimal forecast of the measure \citep{ziegel_coherence_2013}. The non-elicitability of ES brings many problems in estimating and backtesting since the corresponding M-estimation or test statistic is hard to construct directly \citep{acerbi2014back,fissler_expected_2015,fissler_higher_2016}.

\cite{kuan_assessing_2009} noted that expectile is another measure for assessing the tail-related risk; see also \cite{bellini_risk_2017}. Expectile shares both elicitability and coherency. Similar to quantile from asymmetric absolute loss, the expectile with significance level $\tau$ of a random variable $X$ originates from the asymmetric squares loss\citep{newey_asymmetric_1987}
\begin{eqnarray}
\mu_\tau(X)=\arg\underset{\mu}\min \mathbf{E}[\rho_{\tau}(X-\mu)],
\end{eqnarray}
where the asymmetric squares loss $\rho_{\tau}(\cdot)$ is defined as $\rho_{\tau}(r)=|\tau-I(r<0)|r^2$. Without loss of generality, suppose that $\mathbf{E}(X)=0$; by a straightforward calculation, we have
\begin{eqnarray}\label{expectiletoessim}
\mathbf{E}(X|X<\mu_\tau(X))=\left(1+\frac{\tau}{(1-2\tau)F_X(\mu_\tau(X))}\right)\mu_\tau(X).
\end{eqnarray}
Eq.(\ref{expectiletoessim}) indicates the specific relation between expectile and ES. \cite{jones_expectiles_1994} proposed another proposition of expectile, showing that there exists a unique increasing bijective function $h:(0,1)\rightarrow(0,1)$ such that $Q_{\alpha}(X)=\mu_\tau(X)$, when $\tau=h_X(\alpha)$, where $h_X$ is defined as
\begin{eqnarray}\label{htheta}
h_X(\alpha)=\frac{-\alpha Q_{\alpha}(X)+G_X(Q_{\alpha}(X))}{2G_X(Q_{\alpha}(X))+(1-2\alpha)Q_{\alpha}(X)},
\end{eqnarray}
with $G_X(q)=\int_{-\infty}^{q}tdF_X(t)$ the partial moment function of $X$. Eq.(\ref{htheta}) builds the close relation between expectile and quantile.

Based on Eq.(\ref{expectiletoessim}) and Eq.(\ref{htheta}), we can address three tail-related risk simultaneously if we know the mapping $\tau=h_X(\alpha)$. Hence, finding the mapping from $\alpha$ to $\tau$ is an important issue. \cite{taylor_estimating_2008} proposed a grid-search method to determine the corresponding $\tau$ for fixed $\alpha$. This approach may lose estimation accuracy when the historical financial data are insufficient (see the boxplot in Figure 1).

To improve the estimation accuracy, we adopt the empirical likelihood method, which is data-driven and distribution-free, to estimate $\tau$. The empirical likelihood method is an effective and flexible nonparametric method of statistical inference \citep{owen2001empirical}. Maximum empirical likelihood method based on estimating equations is one such method, which is typically used for point estimation \citep{qin_empirical_1994}. The properties of expectile described above opportunely provide us with the following two estimating equations:
\begin{eqnarray}\label{theoel}
\left\{
\begin{aligned}
 &\mathbf{E}[(X-\mu_{\tau}(X)I(X<\mu_{\tau}(X))+\frac{\tau}{1-2\tau}X]-\frac{\tau}{1-2\tau}\mu_{\tau}(X)=0,\\
 &\mathbf{E}[I(X<\mu_{\tau}(X))]-\alpha=0.\\
\end{aligned}
\right.
\end{eqnarray}
from which we construct the maximum empirical likelihood method to estimate $\tau$.

Here, a simple simulation in the i.i.d case is performed to evaluate the performance of the empirical likelihood method. The data are generated from standard normal distribution and Student's distribution with 5 degrees of freedom, and the sample size is set to 1000, 300 and 100. \cite{taylor_estimating_2008}'s grid-search method and the empirical likelihood method are used to estimate the $\tau$ value corresponding to $\alpha=0.05$. We repeat each data generation and estimation procedure 100 times.

The squared error is used as the evaluation criterion to compare the estimation accuracy of the grid-search method and the empirical likelihood method. Figure 1 shows the boxplot of the squared errors of the two methods in different cases. The simulation result convinces us that the empirical likelihood method is a proper choice to solve the $\tau$ selection problem. In addition to the competitive estimation accuracy, the estimation procedure of the empirical likelihood method is more computationally convenient than the grid-search method is.

\begin{figure}[H]
\captionsetup{font={small}}
\centering
\subfigure{
\includegraphics[width=0.3\textwidth]{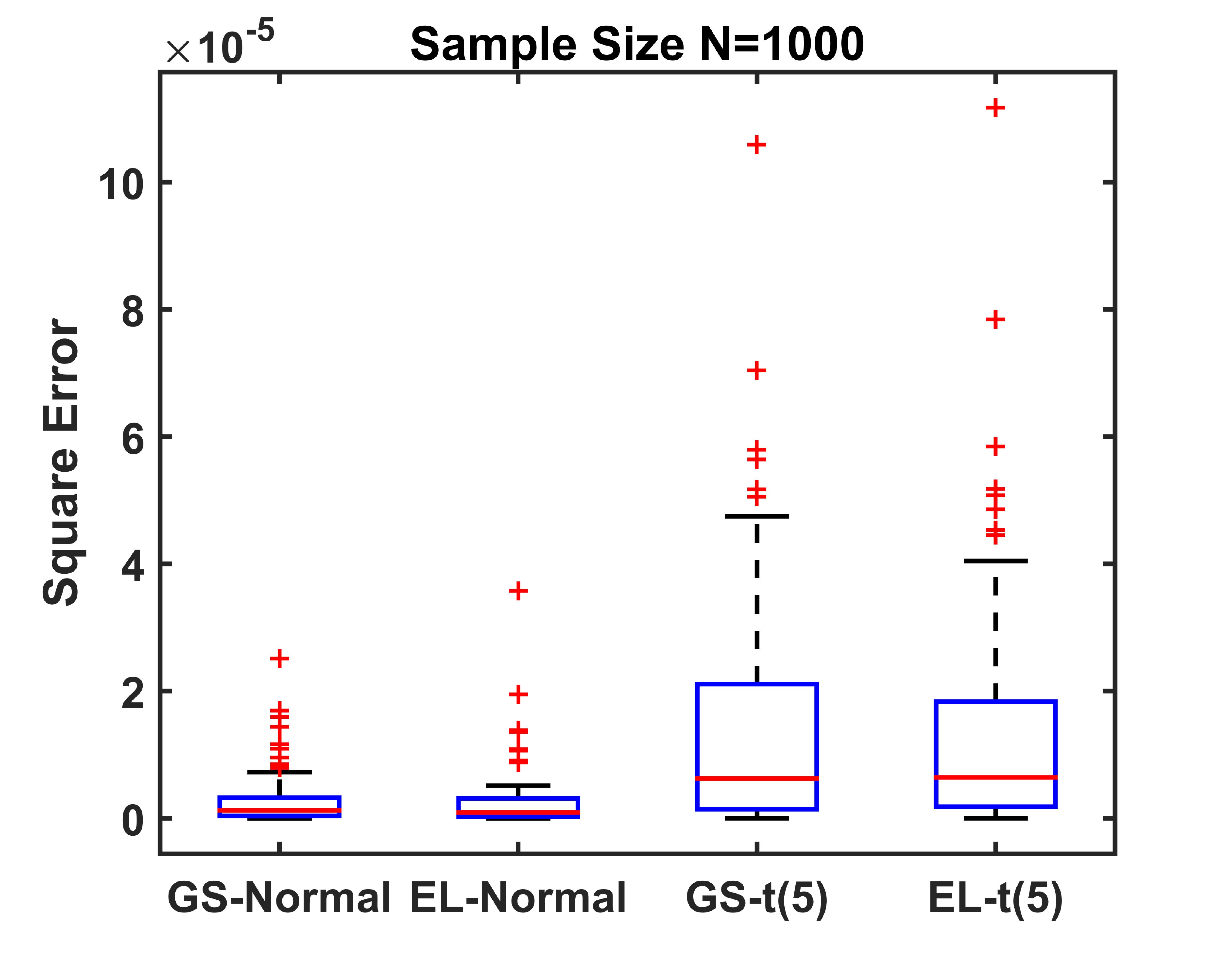}
}
\subfigure{
\includegraphics[width=0.3\textwidth]{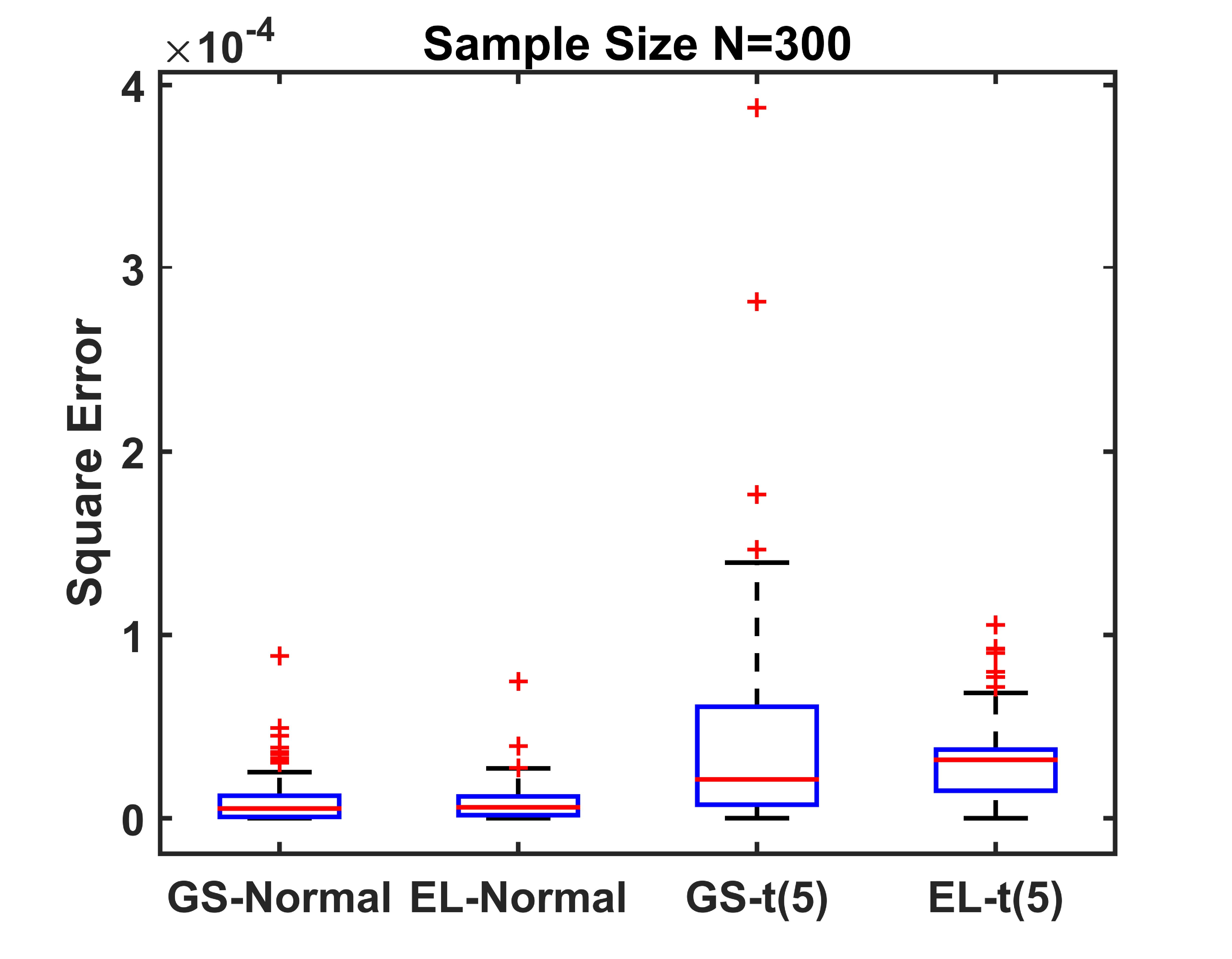}
}
\subfigure{
\includegraphics[width=0.3\textwidth]{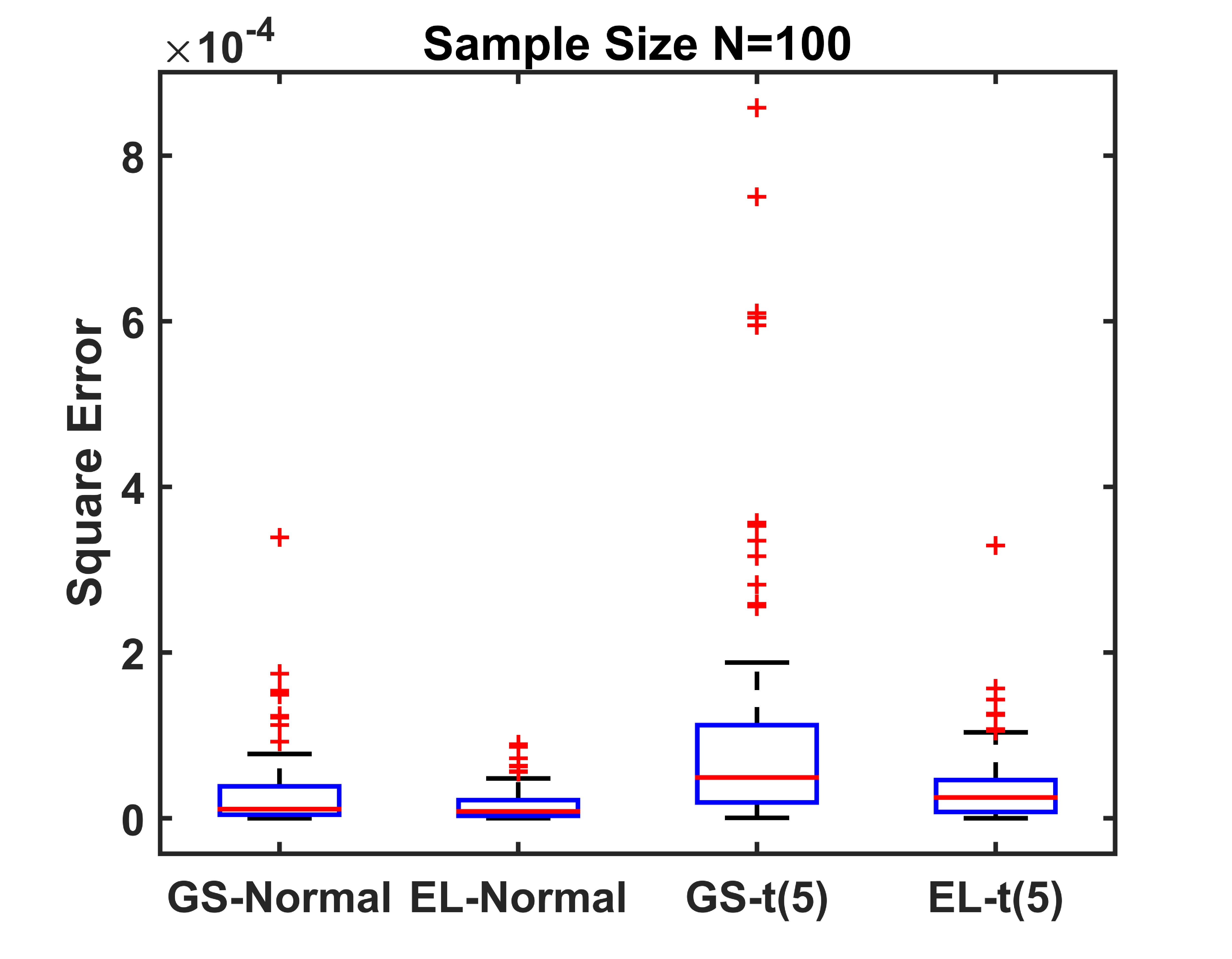}
}
\caption{Boxplot of the squared error of $\tau$ estimation using the grid-search and the empirical likelihood methods}
\end{figure}

\subsection{Conditional tail-related risk under the GARCH framework}

In this section, we focus on how the dynamic structure of conditional tail-related risks stems from the dynamic structure of volatility under the GARCH framework. This analysis inspires us to propose the CALS method, which is an improvement with respect to the composite quantile method \citep{xiao_conditional_2009} for estimating volatility.

Consider a GARCH-type return series $\{Y_t\}_{t\geq1}$ given by
\begin{eqnarray}
Y_t=\sigma_t \cdot\varepsilon_t,
\end{eqnarray}
with natural filtration $\{\mathcal{F}_t\}_{t\geq1}$. The regular assumptions for the GARCH model \citep{berkes_garch_2003, hall_inference_2003} includes the following: $\{\varepsilon_t\}_{t\geq1}$ is an innovation series (i.e. independent
identically distributed) from marginal distribution $F_\varepsilon$ with zero mean and unit variance; $\varepsilon_t$ is independent to $\mathcal{F}_{t-1}$; and the volatility term $\sigma_t$ is measurable with respect to $\mathcal{F}_{t-1}$.

For quantile level $\alpha$ and expectile index $\tau$, denote $Q_\alpha(\varepsilon)$ and $ES_\alpha(\varepsilon)$ as the $\alpha$-th quantile and ES of the marginal distribution $F_\varepsilon$, and let $\mu_\tau(\varepsilon)$ be its $\tau$-th expectile. The conditional VaR, conditional ES and conditional expectile of $\{Y_t\}_{t\geq1}$ are represented as
\begin{eqnarray}
Q_\alpha(Y_{t}|\mathcal{F}_{t-1})=\sigma_{t}\cdot Q_\alpha(\varepsilon),
\end{eqnarray}
\begin{eqnarray}
ES_\alpha(Y_{t}|\mathcal{F}_{t-1})=\sigma_{t}\cdot ES_\alpha(\varepsilon),
\end{eqnarray}
\begin{eqnarray}
\mu_\tau(Y_{t}|\mathcal{F}_{t-1})=\sigma_{t}\cdot \mu_\tau(\varepsilon).
\end{eqnarray}

Following the linear GARCH model in \cite{xiao_conditional_2009}, the volatility $\sigma_t$ has the following dynamic structure:
\begin{eqnarray}\label{yt}
Y_t=\sigma_t \cdot\varepsilon_t,
\end{eqnarray}
\begin{eqnarray}\label{garch}
\sigma_t=\beta_{0}+\sum_{i=1}^{p}\beta_{i}\sigma_{t-i}+\sum_{j=1}^{q}\gamma_{j}|Y_{t-j}|,
\end{eqnarray}
where $\beta_0$ and $\gamma_1,\ldots,\gamma_q$ are positive. Hence, we obtain the autoregressive specification of conditional VaR, conditional ES and conditional expectile as follows:
\begin{eqnarray}\label{varmodel}
Q_\alpha(Y_t|\mathcal{F}_{t-1})=\beta_{0}^{\ast}+\sum_{i=1}^{p}\beta_{i}^{\ast}Q_\alpha(Y_{t-i}|\mathcal{F}_{t-i-1})+\sum_{j=1}^{q}\gamma_{j}^{\ast}|Y_{t-j}|,
\end{eqnarray}
where $\beta_{0}^{\ast}=\beta_{0}Q_\alpha(\varepsilon)$, $\beta_{i}^{\ast}=\beta_{i},i=1,\ldots,p$, $\gamma_{j}^{\ast}=\gamma_{j}Q_\alpha(\varepsilon),j=1,\ldots,q$;
\begin{eqnarray}\label{ESmodel}
ES_{\alpha}(Y_t|\mathcal{F}_{t-1})=\beta_{0}^{\ast\ast}+\sum_{i=1}^{p}\beta_{i}^{\ast\ast}ES_{\alpha}(Y_{t-i}|\mathcal{F}_{t-i-1})+\sum_{j=1}^{q}\gamma_{j}^{\ast\ast}|Y_{t-j}|,
\end{eqnarray}
where $\beta_{0}^{\ast\ast}=\beta_{0}ES_\alpha(\varepsilon)$, $\beta_{i}^{\ast\ast}=\beta_{i},i=1,\ldots,p$, $\gamma_{j}^{\ast\ast}=\gamma_{j}ES_\alpha(\varepsilon),j=1,\ldots,q$;
\begin{eqnarray}\label{caremodel}
\mu_\tau(Y_t|\mathcal{F}_{t-1})=\beta_{0}^{\ast\ast\ast}+\sum_{i=1}^{p}\beta_{i}^{\ast\ast\ast}\mu_\tau(Y_{t-i}|\mathcal{F}_{t-i-1})+\sum_{j=1}^{q}\gamma_{j}^{\ast\ast\ast}|Y_{t-j}|,
\end{eqnarray}
where $\beta_{0}^{\ast\ast\ast}=\beta_{0}\mu_\tau(\varepsilon)$, $\beta_{i}^{\ast\ast\ast}=\beta_{i},i=1,\ldots,p$, $\gamma_{j}^{\ast\ast\ast}=\gamma_{j}\mu_\tau(\varepsilon),j=1,\ldots,q$.

The above transformation shows us that the autoregressive dynamics of tail-related risks stem from the dynamic of volatility. The coefficients $\beta_1, \ldots, \beta_p$ stay unchanged in all of the dynamics, whether for volatility or for different risks. The rest of the coefficients in the above dynamics are variational for different risk dynamics. It can be seen that these coefficients are the products of the original coefficients and the corresponding risks of $F_\varepsilon$. This fact indicates that the dynamic of a tail-related risk consists of two types of information: dynamic information from the volatility and distribution information from the innovation series.

Based on the above analysis, there are two noteworthy points to highlight here:
\begin{enumerate}[(1)]
\item
We show that under the GARCH framework, the dynamics of different tail-related risks share a specific structure with some latent and common coefficients. As discussed above, ES is non-elicitable; thus, the coefficients in Eq. (\ref{ESmodel}) cannot be estimated directly by solving a optimisation function. With the dynamic structures above, we can obtain the ES dynamic from the expectile dynamic based on the properties of expectile shown in Eq. (\ref{expectiletoessim}) - Eq. (\ref{htheta}). Since the innovation series has zero expectation, if we can determine $\tau=h_\varepsilon(\alpha)$ such that $Q_{\alpha}(\varepsilon)=\mu_\tau(\varepsilon)$, then the $\tau$-th expectile specification in Eq. (\ref{caremodel}) is equivalent to the $\alpha$-th quantile specification in Eq. (\ref{varmodel}). Moreover, the $\alpha$ ES specification can be obtained by multiplying both sides of Eq. (\ref{caremodel}) by a constant $c_\varepsilon$, where $c_\varepsilon=1+\frac{\tau}{(1-2\tau)\alpha}=1+\frac{h_\varepsilon(\alpha)}{(1-2 h_\varepsilon(\alpha))\alpha}$. This is the theoretical basis of \cite{taylor_estimating_2008}'s method, which captures the ES dynamic by estimating the coefficients in the expectile dynamics. In our framework, the distribution information of the innovation series should no longer be ignored because the optimal $\tau$ and $c_\varepsilon$ should be determined by the distribution of $\{\varepsilon_t\}_{t\geq1}$.

\item
Another insight is that the two part of coefficients cannot be divided in a single conditional risk dynamic. For example, we cannot capture $\beta_{0}$ or  $Q_\alpha(\varepsilon)$, even if we have a estimator of $\beta^{\ast}_{0}=\beta_{0}Q_\alpha(\varepsilon)$ from a single quantile regression. The problem can be solved if we consider a cluster of conditional risk dynamics, such as in the composite quantile method \citep{xiao_conditional_2009}. It fits several quantile dynamics to obtain dynamic quantile coefficients and use matrix decomposition to estimate the volatility dynamic coefficients. Here, we find that the conditional expectile shares a similar dynamic structure as conditional quantile. Hence, \cite{xiao_conditional_2009}'s method can be extended to a composite expectile form, which is CALS. Later, we show that a further technical adjustment could make this method more computationally efficient.
\end{enumerate}

We have clarified the dynamic structure of the tail-related risks under the GARCH framework. Next, we state some potential issues and our motivation.

\subsection{Motivation: the combination of CALS and empirical likelihood}
In this part, we summarise the fundamental issues in the estimating procedure and our corresponding approaches; this section describes our main motivation and the primary contributions of this work.

\textbf{The first issue is about the determinant of the mapping $\tau=h_\varepsilon(\alpha)$.} Under the GARCH framework, it has been shown that the dynamics of tail-related risks consist of two part of information: dynamic information from the volatility and distribution information from the innovation series. When we obtain the ES dynamics from the expectile dynamics, the distribution information of the innovation series is necessary to determine two important values: $\tau$ and $c_\varepsilon$. \cite{taylor_estimating_2008}'s grid-search method is based on the return series $Y_t$, which will lead to estimation error because the conditional distribution of $Y_t$ with respect to $\mathcal{F}_{t-1}$ is often different depending on the marginal distribution of $Y_t$. Moreover, the time-varying property of the determinant of $\tau=h_\varepsilon(\alpha)$ is noteworthy. \cite{taylor_estimating_2008} selects $\tau$ based on the data from the first moving window and keeps it fixed in the rest of the estimation procedure. It is not guaranteed that the distribution for the return series is not time-varying. In fact, the issue was verified in the empirical study of \cite{kuan_assessing_2009}, who noted that with a fixed probability-level conditional expectile, the corresponding tail probabilities of the conditional quantile are different in-sample and out-of-sample. Hence, a time-varying mapping $\tau=h_\varepsilon(\alpha)$ is more appropriate in practice.

To overcome these issues, we separate the innovation series from the return series, which involving a volatility estimating step first. Then, the empirical likelihood method is used to determine the value of $\tau$ for fixed quantile level $\alpha$ from the separated innovation series. The empirical likelihood method is performed in a rolling manner in each moving window so the determined value of $\tau$ can be updated over time.

\textbf{The second issue occurs in the procedure of estimating the volatility structure.} To separate the innovation series ($\varepsilon_t=\frac{Y_t}{\sigma_t}$) from the return series, an important step is to estimate the volatility structure. The composite quantile method proposed by \cite{xiao_conditional_2009} is an alternative choice to capture the volatility structure of a conditional heteroscedastic time series. Inspired by the analysis in section 2.2, we make some technical adjustment to perfect this method.

First, \cite{xiao_conditional_2009}'s method is extended to a composite expectile form, which is called CALS. The substitution from quantile to expectile is made for computational efficiency since the expectile regression has better properties in computation than quantile regression does; see the details in \cite{waltrup_expectile_2015}. The second adjustment is keeping the separation of two parts of coefficients in the corresponding optimisation function (for example, maintaining the separation of $\beta_{0}$ and $\mu_\tau(\varepsilon)$), rather than taking their product as a single coefficients. This adjustment can avoid the complex computations of matrix decomposition in \cite{xiao_conditional_2009}.

Based on the potential issues noted above, we combine the CALS method and the empirical likelihood method to estimate the conditional VaR, conditional ES and even conditional expectile simultaneously. The main model assumption of our method is the GARCH-type series with i.i.d innovation, which is different from the semiparametric autoregressive model. The main estimation procedure of the proposed method can be outlined as,
\begin{itemize}
\item Estimating the volatility structure of the return series by CLAS, and separating the innovation series from the return series;
\item Determining the mapping $\tau=h_\varepsilon(\alpha)$ by empirical likelihood, and estimating $Q_{\alpha}(\varepsilon)$, $ES_{\alpha}(\varepsilon)$ simultaneously;
\item Estimating conditional tail-related risks by combining the estimations above.
\end{itemize}
The three steps in the outline are detailed in Section 3, corresponding to the three subsections of Section 3.

The contributions of our work can be summarised as follows: First, we use a flexible nonparametric method, empirical likelihood, to determine the mapping $\tau=h_\varepsilon(\alpha)$ from the separated innovation series, which is competitive in terms of estimation accuracy. Second, we extend the composite quantile method to CALS for volatility estimation. To avoid complex computations of matrix decomposition, we maintain the separation of two parts of coefficients in the corresponding optimisation function. The adjusted method is more efficient in computation and has asymptotic properties similar to those of the composite quantile method. Finally, by analysing the dynamic structure of three conditional tail-related risk, we find that they have the same structure with common coefficients. The combinational method allows us to process the coefficients of volatility dynamic and coefficients of distribution separately. The consequent advantage is that we can perform the estimation of conditional VaR, conditional ES and conditional expectile simultaneously.

\section{Combination method of CALS and empirical likelihood}
In the previous section, we have presented an improved proposal: CALS for volatility estimation and empirical likelihood for the determination of $\tau=h_\varepsilon(\alpha)$. In this section, we present more methodological details of CALS and empirical likelihood and then give the complete estimation procedure of the combinational method.

\subsection{Volatility estimation using CALS}
The idea of volatility estimation using composite expectile is enlightened and improved from the methods of composite quantile in \cite{xiao_conditional_2009} and \cite{kai_local_2010}. We use a class of expectile specifications with common parameter constraints to fit the volatility structure.

Consider the following linear GARCH($p$,$q$) model:
\begin{eqnarray}\label{yt2}
Y_t=\sigma_t \cdot\varepsilon_t,
\end{eqnarray}
\begin{eqnarray}\label{garch2}
\sigma_t=\beta_{0}+\sum_{i=1}^{p}\beta_{i}\sigma_{t-i}+\sum_{j=1}^{q}\gamma_{j}|Y_{t-j}|.
\end{eqnarray}
Denote $A(L)=1-\sum_{i=1}^{p}\beta_{i}L^{i}$ and $B(L)=\sum_{j=1}^{q}\gamma_{j}L^{j-1}$ (where $L$ is the lagged operator) satisfying the invertible assumption [B1] (see section 4). With this assumption, we can obtain an ARCH($\infty$) representation of $\sigma_t$,
\begin{eqnarray}\label{ai}
\sigma_t=a_0+\sum_{i=1}^{\infty}a_i|Y_{t-i}|,
\end{eqnarray}
where the coefficients $a_i$ decrease geometrically, which is implied by the assumption [B1] (see detailed discussion in \cite{Koenker2006Quantile}). Without loss of generality, we normalised $a_0=1$ for identification. Substituting the foregoing ARCH($\infty$) representation into (\ref{yt2}) and (\ref{garch2}), we have
\begin{eqnarray}
Y_t=\left(a_0+\sum_{i=1}^{\infty}a_i|Y_{t-i}|\right)\varepsilon_t.
\end{eqnarray}
Denoting the truncation parameter by $m$, we use the following truncated ARCH($m$) model as an approximation of the real model:
\begin{eqnarray}\label{tqar}
Y_t=\left(a_0+\sum_{i=1}^{m}a_i|Y_{t-i}|\right)\varepsilon_t.
\end{eqnarray}

We construct a composite method, CALS, to fit the volatility structure, and it has a class of expectile autoregressive specifications as
\begin{eqnarray}
\mu_{\tau_k}(Y_t|\mathcal{F}_{t-1})=\mu_{\tau_k}(\varepsilon)\left(a_0+\sum_{i=1}^{m}a_i|Y_{t-i}|\right), k=1,\ldots,K,
\end{eqnarray}
where $\{\tau_k\}_{k=1}^{K}$ is a class of the expectile significance level. The advantage of this method is that a class of expectile autoregressive models can fully exploit the potential information about the volatility structure. Here, expectile specifications with different significance levels share the same parameter structure, with common parameters $a_0$ to $a_m$ for the volatility structure and a specific parameter $\mu_{\tau_k}$ for different significance levels. Essentially, it imposes parameter constraints on different expectile specifications, and this is a big difference from the composite quantile method in \cite{xiao_conditional_2009}.

For convenience of expression, we denote the parameters in the CALS formulas above as
\begin{eqnarray}
\theta\triangleq (\vartheta^T,\eta^T)^T,
\end{eqnarray}
where
\begin{eqnarray}
\vartheta=(u_1,\ldots,u_k)^T=(\mu_{\tau_1}(\varepsilon),\ldots,\mu_{\tau_K}(\varepsilon))^T,
\end{eqnarray}
\begin{eqnarray}
\eta=(a_0,a_1,\ldots,a_m)^T,
\end{eqnarray}
and $a_0$ is fixed as 1 for identification, as is its estimator $\tilde{a}_0$. The parameters $\vartheta$ involve the distribution information of innovation sereis, and the parameters $\eta$ involve the dynamic information of volatility structure. Here, as we discussed in Section 2, we maintain them separation in the loss function.

Denote $x_{t,(m)}=(1,|Y_{t-1}|,\ldots,|Y_{t-m}|)^T$, and $w_t=(Y_t,1,|Y_{t-1}|,\ldots,|Y_{t-m}|)^T=(Y_t,x_{t,(m)}^T)^T$. We can estimate these parameters using CALS with the following expression:
\begin{eqnarray}\label{cals}
\tilde{\theta}& = & \arg\min_{\theta}\sum_{t=1}^{n}\rho(w_t,\theta)\\
& \triangleq & \arg\min_{\theta}\sum_{t=1}^{n}\sum_{k=1}^{K}\rho_{\tau_k}(Y_t-\mu_k\eta^{T}x_{t,(m)})
\end{eqnarray}
Then, we can obtain a preliminary estimation of $\sigma_t$ in-sample as
\begin{eqnarray}\label{pesi_sig}
\tilde{\sigma}_t=\tilde{\eta}_n^Tx_{t,(m)}=\tilde{a}_0+\sum_{i=1}^{m}\tilde{a}_i|Y_{t-i}|.
\end{eqnarray}

To improve estimation accuracy, we can refit the GARCH-(p,q) model by least squares since we already have a preliminary estimation of the volatility. Denote the parameters in the GARCH-(p,q) model by $\phi=(\beta_0,\gamma_1,\ldots,\gamma_q,\beta_1,\ldots \beta_p)^T$, which can be estimated by
\begin{eqnarray}\label{hatp}
\hat{\phi}=(\hat{\beta}_0,\hat{\gamma}_1,\ldots,\hat{\gamma}_q,\hat{\beta}_1,\ldots,\hat{\beta}_p)^T=\arg\min_{\phi}\sum_{t}\left(\tilde{\sigma}_t-\beta_{0}-\sum_{i=1}^{p}\beta_{i}\tilde{\sigma}_{t-i}-\sum_{j=1}^{q}\gamma_{j}|Y_{t-j}|\right)^2.
\end{eqnarray}
The corresponding volatility estimation is
\begin{eqnarray}
\hat{\sigma}_t=\hat{\beta}_{0}+\sum_{i=1}^{p}\hat{\beta}_{i}\tilde{\sigma}_{t-i}+\sum_{j=1}^{q}\hat{\gamma}_{j}|Y_{t-j}|.
\end{eqnarray}
\begin{rem}
Compared to the method of \cite{xiao_conditional_2009}, there are several improvements in our method. First, the asymmetric least squares has better computational properties than asymmetric least absolute; see \cite{waltrup_expectile_2015}. Second, the parameters constrained in CALS makes the model be free from the crossing problem, which often occurs in composite methods; see details in \cite{waltrup_expectile_2015}. Finally, it also avoids the complex matrix decomposition in \cite{xiao_conditional_2009}, making the method more computationally efficient. With lower computational complexity, the proposed method still shares similar asymptotic properties with that in \cite{xiao_conditional_2009}.
\end{rem}

\subsection{Empirical likelihood for determining $\tau=h_\varepsilon(\alpha)$}
Having obtained the estimation of volatility in-sample, a series of estimated innovation can be obtained by
\begin{eqnarray}\label{tvar}
\hat{\varepsilon}_t=\frac{Y_t}{\hat{\sigma}_t},
\end{eqnarray}
where $t=m+1,\ldots,n$. We determine the corresponding $\tau$ for fixed $\alpha$ based on the series of estimated innovation by the method of empirical likelihood since $\tau=h_\varepsilon(\alpha)$ depends on the marginal distribution of noise process $\{\varepsilon_t\}_{t\geq1}$.

We use the maximum empirical likelihood method to determine $\tau=h_\varepsilon(\alpha)$, which is theoretically based on the properties of expectile stated in Eq. (\ref{theoel}). Considering the propositions for innovation series, we have
\begin{eqnarray}
\left\{
\begin{aligned}
 &E[(\varepsilon-\mu_{\tau}(\varepsilon))I(\varepsilon<\mu_{\tau}(\varepsilon))+\frac{\tau}{1-2\tau}\varepsilon]-\frac{\tau}{1-2\tau}\mu_{\tau}(\varepsilon)=0,\\
 &E[I(\varepsilon<\mu_{\tau}(\varepsilon))]-\alpha=0.\\
\end{aligned}
\right.
\end{eqnarray}
With the equations above, for a fixed $\alpha$, we can construct the empirical likelihood function and maximum empirical likelihood estimation of $(\mu_{\tau}(\varepsilon),\tau)$. For notational convenience, in this subsection, we denote the true value of $(\mu_{\tau}(\varepsilon),\tau)$ by $(\mu_0,\tau_0)$, and the notation $\tau$ will be used in the optimisation function of empirical likelihood.

Suppose that $\hat{\varepsilon}_1,\ldots,\hat{\varepsilon}_n$ are the estimated innovation from (\ref{tvar}). For $i=1,\ldots,n$ and fixed $\alpha$, let
\begin{eqnarray}
W_i(\mu,\tau)=\left(W_{i1}(\mu,\tau),\ W_{i2}(\mu,\tau)\right)^T=\left((\hat{\varepsilon}_i-\mu)I(\hat{\varepsilon}_i<\mu)+\frac{\tau}{1-2\tau}\hat{\varepsilon}_i-\frac{\tau}{1-2\tau}\mu,\quad I(\hat{\varepsilon}_i<\mu)-\alpha\right)^T.
\end{eqnarray}
Then, the  empirical likelihood function of $(\mu,\tau)$ can be expressed as
\begin{eqnarray}
L(\mu,\tau)=\sup\{\prod_{i=1}^{n}(np_i)|p_i>0,i=1,2,\ldots,n,\sum_{i=1}^{n}p_i=1,\sum_{i=1}^{n}p_i W_i(\mu,\tau)=0\}.
\end{eqnarray}
Making use of Lagrange multipliers, we can obtain
\begin{eqnarray}\label{inclem}
l(\mu,\tau)\triangleq-2\log L(\mu,\tau)=2\sum_{i=1}^{n}\log(1+\lambda^T W_i(\mu,\tau)),
\end{eqnarray}
where $\lambda=\lambda(\mu,\tau)$ is a two-dimensional vector associated with $(\mu,\tau)$ but has no explicit expression. The relationship of $\lambda$ and $(\mu,\tau)$ is as follows:
\begin{eqnarray}
\sum_{i=1}^{n}\frac{W_i(\mu,\tau)}{1+\lambda^T W_i(\mu,\tau)}=0.
\end{eqnarray}
Next, we can obtain the maximum empirical likelihood estimate for $(\mu_0,\tau_0)$, which is defined as
\begin{eqnarray}\label{hmutau}
(\hat{\mu},\hat{\tau})=\arg\min_{(\mu,\tau)} l(\mu,\tau).
\end{eqnarray}

Here, we use the estimated innovation series $\{\hat{\varepsilon}_t\}_{t\leq n}$ to estimate $\tau$ and $\mu_{\tau}(\varepsilon)$ for fixed $\alpha$ via the empirical likelihood method mentioned above. So far, we have described the method for determining the corresponding $\tau$. It is more reasonable than the method of \citet{taylor_estimating_2008} since our method authentically uses information about the conditional distribution of $Y_t$. Furthermore, the empirical likelihood method is completely distribution-assumption-free and data-driven.

\subsection{Estimating conditional tail-related risks}
At the end of the section, we summarise the combination method of CALS and empirical likelihood for estimating conditional tail-related risks. We are going to estimate the conditional VaR and ES of series conditional on information prior to time $T$. Given a suitable length of moving window $n$, we obtain $\hat{\sigma}_t$ from the observations in the moving window by the method of CALS. Additionally, the estimated innovation series is obtained by
\begin{eqnarray}
\hat{\varepsilon}_t=\frac{Y_t}{\hat{\sigma}_t},
\end{eqnarray}
where $t=m+1,\ldots,n$. From this series of $\hat{\varepsilon}_t$, we obtain the estimation $(\hat{\mu},\hat{\tau})$ using empirical likelihood. Then, the $\alpha$-quantile estimator of $\varepsilon$ is represented as
\begin{eqnarray}\label{varpre1}
\widehat{Q}_\alpha(\varepsilon)=\hat{\mu},
\end{eqnarray}
and the corresponding estimator of ES is
\begin{eqnarray}\label{espre1}
\widehat{ES}_\alpha(\varepsilon)=\left(1+\frac{\hat{\tau}}{(1-2\hat{\tau})\alpha}\right)\hat{\mu}-\frac{\hat{\tau}}{(n-m)(1-2\hat{\tau})\alpha}\sum_{t}\hat{\varepsilon}_t.
\end{eqnarray}
The conditional tail-related risks of $Y_T$ prediction can be obtained from the product of volatility prediction and tail-related risks of the innovation series. Since we have a preliminary estimation $\tilde{\sigma}_t$, a rational choice is predicting the conditional tail-related risks of $Y_T$ by
\begin{eqnarray}\label{eq50}
\widetilde{Q}_\alpha(Y_T|\mathcal{F}_{T-1})=\tilde{\sigma}_T\cdot\widehat{Q}_\alpha(\varepsilon),
\end{eqnarray}
\begin{eqnarray}\label{eq51}
\widetilde{ES}_\alpha(Y_T|\mathcal{F}_{T-1})=\tilde{\sigma}_T\cdot\widehat{ES}_\alpha(\varepsilon),
\end{eqnarray}
where $\tilde{\sigma}_T=\tilde{a}_0+\sum_{i=1}^{m}\tilde{a}_i|Y_{T-i}|$. Additionally, we can predicate the conditional tail-related risks based on another volatility estimation $\hat{\sigma}_T$ as follows:
\begin{eqnarray}\label{eq52}
\widehat{Q}_\alpha(Y_T|\mathcal{F}_{T-1})=\hat{\sigma}_T\cdot\widehat{Q}_\alpha(\varepsilon),
\end{eqnarray}
\begin{eqnarray}\label{eq53}
\widehat{ES}_\alpha(Y_T|\mathcal{F}_{T-1})=\hat{\sigma}_T\cdot\widehat{ES}_\alpha(\varepsilon),
\end{eqnarray}
where $\hat{\sigma}_T=\hat{\beta}_{0}+\sum_{i=1}^{p}\hat{\beta}_{i}\tilde{\sigma}_{T-i}+\sum_{j=1}^{q}\hat{\gamma}_{j}|Y_{T-j}|.$ In the simulation section, we show that both of two estimation are competitive, but the 'hat' one (Eqs. (\ref{eq52}) - (\ref{eq53})) outperforms the 'tilde' one (Eqs. (\ref{eq50}) - (\ref{eq51})).

So far, we have described the procedure for estimating the conditional tail-related risk by using our proposed method. Now, let us present some details about the rule of thumb for selecting the tuning parameters in our method. The length of moving window, $n$, must be determined since it is crucial to the estimation. Consider the overall asymptotic properties of the method, a bigger $n$ is preferable. However, in reality, financial time series are model-changing frequently, in terms of both the heteroscedasticity structure and noise distribution. An overly large $n$ may lead to undesirable model errors. The selection of $n$ is also sensitive to the quantile level $\alpha$. Based on simulations, a moving window with length from 500 to 1000 is suitable for $\alpha$ is not too extreme, and longer moving windows are necessary for situations with a more extreme $\alpha$.

The truncation parameter $m=m(n)$ is a value associated with the sample size $n$. \citet{xiao_conditional_2009} proposed that $m$ should be a sufficiently large constant multiple of $\log(n)$ to ensure that the approximation error of $\tilde{\sigma}_t$ is sufficiently small. As the method for  preliminary estimation $\tilde{\sigma}_t$ in our paper is essentially similar to the method of \citet{xiao_conditional_2009}, we follow its selection, with $m(n)=O(n^{\frac{1}{4}})$.

The number of expectile specifications, $K$, and the corresponding expectile index $\{\tau_k\}_{k=1}^{K}$ in CALS must also be determined by analysts. Usually, we choose a uniform grid over the interval $(0,1)$ as the class of expectile index $\{\tau_k\}_{k=1}^{K}$. Although a larger K will improve the estimation accuracy, the computational expense along with an increase in K and robustness of estimation when some $\tau_k$ approaching 0 or 1 should be considered. As verified via simulations, a uniform grid over the interval $(0,1)$ with a length $K$ from 9 to 19 is an appropriate choice for not-too-extreme $\alpha$.

\section{Asymptotic properties of the combination estimation}
In this section, we state the asymptotic properties of the proposed methods as theorems, and the corresponding proofs are all presented in the appendix.
\subsection{Asymptotic properties of CALS estimation}
Let $\theta_0=(\vartheta_0^T,\eta_0^T)^T$ be the minimizer of $\mathbf{E}[\rho(w_t,\theta)]=\mathbf{E}[\sum_{k=1}^{K}\varphi_k(w_t,\theta)v_k^2(w_t,\theta)]$. Consider the estimation of $\theta_0$ derived from the CALS in Eq. (\ref{cals}). The consistency and asymptotic normality of the CALS estimator are given by \emph{Theorem 4.1} and \emph{Theorem 4.2}. We first give some necessary conditions for these results.

\begin{ass}
A1. $w_t=(Y_t,x_{t,(m)}^T)^T$ is strictly stationary and ergodic and has the probability density function $f_w(w_t) = g(y_t|x_t)h(x_t)$ with respect to the measure $\upsilon_w=\pi\times \upsilon_x$, where $f_w(w_t)$ is continuous in $y_t$ for almost all $x_t$ and $\pi$ denotes the Lebesque measure on the $\mathbb{R}$.
\end{ass}

\begin{ass}
A2. There is a $\delta>0$ such that $\int|w_t|^{4+\delta}g(y_t|x_t)h(x_t)d\upsilon_w<\infty$, where $|\cdot|$ denotes the infinite norm in this article.
\end{ass}

\begin{ass}
A3. $\theta\in \Theta \subset \mathbb{R}^{K+m}$, where $\Theta$ is compact.
\end{ass}

\begin{ass}
A4. $\mathbf{E}[x_{t,(m)} x_{t,(m)}^T]$ is nonsingular.
\end{ass}

These assumptions are common in asymmetric least squares regression; see \cite{newey_asymmetric_1987}. Under these regularity conditions, we have the following theorems about the asymptotic properties of CALS estimation.
\begin{thm}\label{thm1}
Under Assumptions [A1]-[A3], there is a unique minimiser, $\theta_0$, of the object function $\mathbf{E}[\rho(w_t,\theta)]$, and the CALS estimator $\tilde{\theta}$ satisfies, $\tilde{\theta}\stackrel{p}{\rightarrow}\theta_0$, as $n\rightarrow\infty$.
\end{thm}

\begin{thm}\label{thm2}
Under assumptions [A1]-[A4], $$\sqrt{n}(\tilde{\theta}-\theta_0)\stackrel{d}{\rightarrow}N(\mathbf{0},\Xi),$$ as $n\rightarrow\infty$, where $\Xi={\Sigma}^{-1}{\Omega}{\Sigma}^{-1}$, with
\begin{eqnarray*}
\Omega=\lim_{n\rightarrow \infty}\mathbf{Var}\left[\frac{1}{\sqrt{n}}\sum_{t}\frac{\partial \rho(w_t,\theta)}{\partial \theta}|_{\theta=\theta_0}\right],
\end{eqnarray*}
and
\begin{eqnarray*}
\Sigma=\mathbf{E}\left[\frac{\partial^2 \rho(w_t,\theta)}{\partial\theta\partial\theta'}|_{\theta=\theta_0}\right].
\end{eqnarray*}
More detailed presentations of ${\Omega}$ and ${\Sigma}$ can be found in the appendix.
\end{thm}

In fact, for volatility estimation, we only need part of the parameters in $\theta$, which are $\eta=(a_0,a_1,\ldots,a_m)^T$. We rewrite the asymptotic property of these parameters as in the following corollary.
\begin{cor}\label{cor1}
Under Assumptions [A1]-[A4], the CALS estimation of $\eta$ satisfies $$\tilde{\eta}\stackrel{p}{\rightarrow}\eta_0,$$ and $$\sqrt{n}(\tilde{\eta}-\eta_0)\stackrel{d}{\rightarrow}N(\mathbf{0},\Xi_{22}),$$ as $n\rightarrow\infty$, where $\Xi_{22}$ is the principal submatrix of $\Xi$ from $(K+1)$th line to $(K+m+1)$th line. Alternatively, it can be presented as $\Xi_{22}={\Omega}_{22}^{-1}{\Sigma}_{22}{\Omega}_{22}^{-1}$, with $${\Omega}_{22}=\lim_{n\rightarrow \infty}\mathbf{Var}\left[\frac{1}{\sqrt{n}}\sum_{t}\frac{\partial \rho(w_t,\theta)}{\partial \eta}|_{\theta=\theta_0}\right],$$ and $${\Sigma}_{22}=\mathbf{E}\left[\frac{\partial^2 \rho(w_t,\theta)}{\partial\eta\partial\eta'}|_{\theta=\theta_0}\right].$$
\end{cor}

Before we provide the asymptotic properties of volatility estimation $\tilde{\sigma}_t$ and $\hat{\sigma}_t$, we should discuss the error from approximating linear GARCH($p$,$q$) by the truncated ARCH($m$) and determine suitable truncation parameters $m(n)$. Following the conclusion of \cite{xiao_conditional_2009}, we present some necessary assumptions to bound the error from this part of approximating.
\begin{ass}
B1. The polynomials $A(L)=1-\sum_{i=1}^{p}\beta_{i}L^{i}$ and $B(L)=\sum_{j=1}^{q}\gamma_{j}L^{j-1}$, where $\beta_0$ and $\gamma_{1},\ldots,\gamma_{q}$ are positive, have no common zero points; $A(z)\neq 0$, for $|z|\leq 1$; and $B(z)\neq 0$, for $|z|\leq 1$.
\end{ass}

\begin{ass}
B2. The truncation parameter m satisfies $m(n)=c\log n$ for some constant $c>0$.
\end{ass}

Under Assumption B1, $A(L)$ is invertible, and the parameters $a_i$ in (\ref{ai}) decrease at a geometric rate. As a consequence, we have the following proposition given by \cite{xiao_conditional_2009}.
\begin{prop}\label{prop1}
Under Assumptions [B1]-[B2], there exists a positive constant $b<1$ such that $\sigma_t$ has approximation as $\sigma_t=x_{t,(m)}^T\eta_0+O_p(b^m)$. If we choose the constant $c$ in Assumption. B2 as $c=\frac{1}{-\log b}$, then $\sigma_t=x_{t,(m)}^T\eta_0+O_p(1/n)$.

\end{prop}

With the conclusions above, we can obtain the asymptotic properties of the preliminary estimation, $\tilde{\sigma}_t$.
\begin{cor}\label{cor2}
Under Assumptions [A1]-[A4] and [B1]-[B2], conditional on the information prior to time $t$, the preliminary estimation, $\tilde{\sigma}_t$, has the following asymptotic properties,
$$\tilde{\sigma}_{t}\stackrel{p}{\rightarrow}\sigma_t,$$ and $$\sqrt{n}(\tilde{\sigma}_t-\sigma_t)\stackrel{d}{\rightarrow}N(0,\varpi^{a}_t),$$ as $n\rightarrow\infty$, where $\varpi^{a}_t=x_{t,(m)}^T\Xi_{22}x_{t,(m)}$.
\end{cor}

To make a consistent one-step post-sample prediction of condition variance, we should discuss the asymptotic properties of $\hat{\phi}$ in (\ref{hatp}). Let us present some notation before the discussion.

Let $z_t=\left[1,|Y_{t-1}|,\ldots,|Y_{t-q}|,\sigma_{t-1},\ldots,\sigma_{t-p}\right]^T$, and denote $$z_t(\tilde{\eta})=\left[1,|Y_{t-1}|,\ldots,|Y_{t-q}|,\sigma_{t-1}(\tilde{\eta}),\ldots,\sigma_{t-p}(\tilde{\eta})\right]^T,$$ since $\tilde{\sigma}_t$ can be expressed as $\tilde{\sigma}_t=\sigma_{t}(\tilde{\eta})=x_{t,(m)}^T\tilde{\eta}$. Correspondingly, we write
$$z_t(\eta_0)=\left[1,|Y_{t-1}|,\ldots,|Y_{t-q}|,\sigma_{t-1}(\eta_0),\ldots,\sigma_{t-p}(\eta_0)\right]^T.$$
Then, the estimator of $\phi_0$ from (\ref{hatp}) can be rewritten as
\begin{eqnarray}\label{ols}
\hat{\phi}=\arg\min_{\phi}\frac{1}{n}\sum_{t}\left(\sigma_{t}(\tilde{\eta})-z_t^T(\tilde{\eta})\phi\right)^2,
\end{eqnarray}
and has the following limiting behaviour.

\begin{thm}\label{thm3}
Under Assumptions [A1]-[A4] and [B1]-[B2], the estimator of $\phi$ from Eq. (\ref{hatp}) has the following asymptotic properties:
\begin{eqnarray}\hat{\phi}\stackrel{p}{\rightarrow}\phi_0,\end{eqnarray}
\begin{eqnarray}\sqrt{n}(\hat{\phi}-\phi_0)\stackrel{d}{\rightarrow}N(0,\Xi_{\phi}),\end{eqnarray}
as $n\rightarrow\infty$. $\Xi_{\phi}$ can be expressed as
\begin{eqnarray}\label{eq56}
\Xi_{\phi}=\Gamma_{10}^{-1}\Gamma_{20}\Xi_{22}(\Gamma_{10}^{-1}\Gamma_{20})^T\end{eqnarray}
with $\Gamma_{10}=\mathbf{E}[z_t(\eta_0)z_t^T(\eta_0)]$ and $\Gamma_{20}=\mathbf{E}\left[z_t(\eta)\left(\frac{d\sigma_t(\eta)}{d\eta^T}-\sum_{j=1}^{p}\beta_j\frac{d\sigma_{t-j}(\eta)}{d\eta^T}\right)\right]|_{\phi=\phi_0,\eta=\eta_0}$.
\end{thm}

Similar to corollary \ref{cor2}, we have the asymptotic properties of $\hat{\sigma}_t$.
\begin{cor}\label{cor3}
Under Assumptions [A1]-[A4] and [B1]-[B2] and conditional on information prior to time $t$, $\hat{\sigma}_t$ has the following asymptotic properties,
$$\hat{\sigma}_{t}\stackrel{p}{\rightarrow}\sigma_t,$$ and $$\sqrt{n}(\hat{\sigma}_t-\sigma_t)\stackrel{d}{\rightarrow}N(0,\varpi^{b}_t),$$ as $n\rightarrow\infty$, where $\varpi^{b}_t=z_{t}^T\Xi_{\phi}z_{t}$, and can be approximated by $\tilde{\varpi}^{b}_t=z_{t}^T(\tilde{\eta})\Xi_{\phi}z_{t}(\tilde{\eta})$.
\end{cor}


\subsection{Asymptotic properties of estimation of the conditional tail-related risk}
Let us turn to estimate the conditional tail-related risk after providing the asymptotic properties of CALS estimation. Since the estimations of the conditional tail-related risks are a combination of the estimation of volatility and the empirical likelihood estimation of $\widehat{Q}_\alpha(\varepsilon)$ and $\widehat{ES}_\alpha(\varepsilon)$, the latter's asymptotic properties must be discussed. These asymptotic results also require some assumptions about the distribution of the innovation series $\{\varepsilon_t\}_{t\geq1}$, which is not strict for most of the common distribution.

\begin{ass}
C1. In each moving window, the innovation series $\{\varepsilon_t\}_{t\geq1}$ is an independent identically distributed random sample with distribution $F_\varepsilon$.
\end{ass}

\begin{ass}
C2. $F_\varepsilon$ has expectation $\mathbf{E}(\varepsilon)=0$ and finite secondary moment $\mathbf{E}[\varepsilon^2]<\infty$. Its derivative $f_\varepsilon$ is bounded and satisfies $f_\varepsilon(\mu_{\tau})\neq0$.
\end{ass}

With these two assumptions and the assumption mentioned above, we can obtain the following lemma about the empirical distribution of the estimated innovation $\{\hat{\varepsilon}_t\}_{t>m}$.
\begin{lem}\label{empcon}
Suppose that $\hat{F}_n(x)$ is the empirical distribution of the estimated innovation $\{\hat{\varepsilon}_t\}_{t>m}$. Under Assumptions [A1]-[A4], [B1]-[B2] and [C1]-[C2], for any given $C>0$, we have
\begin{eqnarray}
\sup_{|x|\leq C}|\hat{F}_n(x)-F_\varepsilon(x)|\leq O_p(n^{-1/2}),
\end{eqnarray}
as $n\rightarrow\infty$.
\end{lem}

Lemma \ref{empcon} plays an important role in deriving the asymptotic property of the empirical likelihood estimation, which actually implies that the empirical distribution of the estimated innovation has similar convergent properties as the empirical distribution of an i.i.d sample. Based on this lemma, the asymptotic property of empirical likelihood estimation $(\hat{\mu}_{\tau},\hat{\tau})$ is established by the following theorem.
\begin{thm}\label{th4}
For the empirical likelihood estimation $(\hat{\mu}_{\tau},\hat{\tau})$ from (\ref{hmutau}), under Assumptions [A1]-[A4], [B1]-[B2] and [C1]-[C2], when $\tau\neq0$, it follows that
\begin{eqnarray}
\hat{\mu}_{\tau}\stackrel{p}{\rightarrow}\mu_{\tau},\quad \hat{\tau}\stackrel{p}{\rightarrow}\tau,
\end{eqnarray}
\begin{eqnarray}
\sqrt{n}\left(\begin{array}{c}\hat{\mu}_{\tau}-\mu_{\tau} \\ \hat{\tau}-\tau \end{array}\right)\stackrel{d}{\rightarrow}N\left({\bf 0},\Sigma_1^{-1}\Sigma_0(\Sigma_1^{-1})^T\right),
\end{eqnarray}
as $n\rightarrow\infty$, where $\Sigma_1=\left[\begin{array}{cc}-(F_\varepsilon(\mu_{\tau})+\frac{\tau}{1-2\tau})&\frac{-\mu_{\tau}}{(1-2\tau)^2}\\f_\varepsilon(\mu_{\tau})
&0\end{array}\right]$, $\Sigma_0=\left[\begin{array}{cc}\sigma_{1}^{2}&\sigma_{12}\\\sigma_{21}&\sigma_{2}^{2}\end{array}\right]$ with
\begin{eqnarray*}
\sigma_{1}^{2}&=&\mathbf{E}\left[(\varepsilon-\mu_{\tau})I(\varepsilon<\mu_{\tau})+\frac{\tau}{1-2\tau}(\varepsilon-\mu_{\tau})\right]^2,\\
\sigma_{12}&=&\sigma_{21}=\mathbf{E}\left[\left((\varepsilon-\mu_{\tau})I(\varepsilon<\mu_{\tau})+\frac{\tau}{1-2\tau}(\varepsilon-\mu_{\tau})\right)(I(\varepsilon<\mu_{\tau})-\alpha)\right],\\
\sigma_{2}^{2}&=&\mathbf{E}[I(\varepsilon<\mu_{\tau})-\alpha]^2=\alpha(1-\alpha).
\end{eqnarray*}
\end{thm}

Eventually, we obtain the asymptotic properties of the combination estimation for both the conditional VaR and conditional ES.
\begin{thm}\label{th5}
Under  Assumptions [A1]-[A4], [B1]-[B2] and [C1]-[C2], when $\tau\neq0$, the conditional tail-related estimation (\ref{eq51}) and (\ref{eq52}) has the following asymptotic properties:
\begin{eqnarray}
\widehat{Q}_\alpha(Y_T|\mathcal{F}_{T-1})\stackrel{p}{\rightarrow}Q_\alpha(Y_T|\mathcal{F}_{T-1}),\quad \widehat{ES}_\alpha(Y_T|\mathcal{F}_{T-1})\stackrel{p}{\rightarrow}ES_\alpha(Y_T|\mathcal{F}_{T-1}),
\end{eqnarray}
\begin{eqnarray}
\sqrt{n}(\widehat{Q}_\alpha(Y_T|\mathcal{F}_{T-1})-Q_\alpha(Y_T|\mathcal{F}_{T-1}))\stackrel{d}{\rightarrow}N\left(0,AVar(\widehat{Q}_\alpha(Y_T|\mathcal{F}_{T-1}))\right),
\end{eqnarray}
\begin{eqnarray}
\sqrt{n}(\widehat{ES}_\alpha(Y_T|\mathcal{F}_{T-1})-ES_\alpha(Y_T|\mathcal{F}_{T-1}))\stackrel{d}{\rightarrow}N\left(0,AVar(\widehat{ES}_\alpha(Y_T|\mathcal{F}_{T-1}))\right),
\end{eqnarray}
as $n\rightarrow\infty$. The asymptotic variance of $\widehat{Q}_\alpha(Y_T|\mathcal{F}_{T-1})$ and $\widehat{ES}_\alpha(Y_T|\mathcal{F}_{T-1})$ are as follows:
\begin{eqnarray}
AVar(\widehat{Q}_\alpha(Y_T|\mathcal{F}_{T-1}))=(\sigma_T)^2\Lambda_1\Xi_{\phi}\Lambda_1^T+(Q_\alpha(\varepsilon))^2z_T^{T}\Xi_{\phi}z_T+2\sigma_T Q_\alpha(\varepsilon)\Lambda_1\Xi_{\phi}z_T,\\
AVar(\widehat{ES}_\alpha(Y_T|\mathcal{F}_{T-1}))=(\sigma_T)^2\Lambda_2\Xi_{\phi}\Lambda_2^T+(ES_\alpha(\varepsilon))^2z_T^{T}\Xi_{\phi}z_T+2\sigma_T ES_\alpha(\varepsilon)\Lambda_2\Xi_{\phi}z_T,
\end{eqnarray}
where the formula of $\Lambda_1$ and $\Lambda_2$ are provided in the appendix. These two asymptotic variances can be approximated by the plug-in method.
\end{thm}

In this section, we have discussed the theoretical properties of our method comprehensively, including the asymptotic properties of CALS estimation, the empirical estimation for determining the $\tau$ and the combination estimation for conditional risks.

\section{Simulation results}
In this section, we present the results of a comparison between our method and some alternative methods. Both the conditional VaR estimation and conditional ES estimation will be involved.

The data generation process is designed as in \cite{xiao_conditional_2009} to test the estimation performance of our proposal for different GARCH coefficients and innovation distributions. Specifically, we generate linear GARCH(1,1) samples with different coefficients. The choices of GARCH coefficients include three cases:
\begin{itemize}
\item Case 1: $\beta_0=0.1, \beta_1=0.5, \gamma_1=0.3;$
\item Case 2: $\beta_0=0.1, \beta_1=0.8, \gamma_1=0.1;$
\item Case 3: $\beta_0=0.1, \beta_1=0.9, \gamma_1=0.05;$
\end{itemize}
These three cases are closer to boundary of invertible condition (Assumption B1 in section 3) of the GARCH process in order. As we know, the closer to boundary of invertible condition, the more difficult it is to fit the GARCH process since the process will be nearly integrated \citep{xiao_conditional_2009}. Here, these cases of coefficients are designed to examine the performance of our proposal in the nearly non-stationary situation. The innovation series are generated from i.i.d standard normal series or i.i.d student t series with 4 degrees of freedom. We would like to test the method in a common case (normal series) and a heavy-tailed case (t(4) series).

Each time, a sample with size 550 is generated from the above GARCH(1,1) process. We divide each data series into two parts: 500 observations as the in-sample and 50 observations as the post-sample. Each data generation process is repeated 1000 times. Finally, we take the average Bias and RMSE (root mean square error) as criteria to evaluate the estimation procedure.

For comparison purposes, we introduce the following alternative methods for conditional VaR and conditional ES estimation.
\begin{itemize}
\item GGARCH: The GARCH(1,1) model with Gaussian innovation assumption, estimated by maximum likelihood.
\item TGARCH: The GARCH(1,1) model with Student's t innovation assumption, estimated by maximum likelihood.
\item CAViar \footnote{The method is available only for conditional VaR estimation\label{fn1}}: The conditional autoregressive value at risk model proposed by \cite{engle_caviar_2004}, with the number of grid points
chosen to be $n$.
\item CARE: The conditional autoregressive expectile method proposed by \cite{taylor_estimating_2008}, with the step size of grid for selecting optimal $\tau$ chosen to be 0.0001.
\item QGARCH \textsuperscript{\ref{fn1}}: The quantile autoregression sieve approximation proposed by \cite{xiao_conditional_2009}. Here, we used is the iteration algorithm in that article with truncation parameter $m=13$ and 19 equally spaced position quantiles, where $\tau_k = 5k\%$ and $k=1,\ldots,19$.
\item QGARCH-EE: Estimate volatility by QGARCH and obtain ES estimation from the empirical distribution of standardized returns.
\item QGARCH-EL: Estimate volatility by QGARCH and obtain ES estimation by empirical likelihood from the standardized returns.
\item CALS-EL: The method proposed in this article combined with composite asymmetric least square and empirical likelihood, where CALS-EL1 using Eq. (\ref{eq50}) and Eq. (\ref{eq51}), and CALS-EL2 using Eq. (\ref{eq52}) and Eq. (\ref{eq53}). They both use the truncation parameter $m=13$ and 19 equally spaced position expectiles ($\tau_k = 5k\%, k=1,\ldots,19$).
\end{itemize}

\begin{rem}
The CAViar and QGARCH method is proposed only for conditional VaR estimation but not for conditional ES estimation. Since QGARCH is an important benchmark of the proposed method, we combine it with the empirical distribution-based method and empirical likelihood method, such that the combinational methods QGARCH-EE and QGARCH-EL become available for conditional ES estimation.
\end{rem}

Here, we list the simulation results of different methods for $\alpha=0.95$. Table 1 and Table 2 show the result of conditional VaR estimation and conditional ES estimation, respectively.

\begin{table}[htbp]
\caption{Simulation results of conditional VaR estimation}
\scriptsize
\begin{tabular}{lcccc|cccc|ccccc}
\hline
\multicolumn{13}{c}{VaR Estimation}\\
\hline
\multirow{3}{*}{Method} & \multicolumn{4}{c}{Case 1}&\multicolumn{4}{c}{Case 2}&\multicolumn{4}{c}{Case 3}\\ \cline{2-5}\cline{6-9}\cline{10-13} & \multicolumn{2}{c}{normal} & \multicolumn{2}{c}{t(4)} & \multicolumn{2}{c}{normal} & \multicolumn{2}{c}{t(4)} & \multicolumn{2}{c}{normal} & \multicolumn{2}{c}{t(4)}\\ \cline{2-3}\cline{4-5}\cline{6-7}\cline{8-9}\cline{10-11}\cline{12-13}
& Bias & RMSE & Bias & RMSE & Bias & RMSE & Bias & RMSE & Bias & RMSE & Bias & RMSE \\
\hline
GGARCH & 0.1324 & 0.1825 & 0.5529 & 0.6318 & 0.2327& 0.3165 & 1.7966 & 2.0507 & 0.5946& 0.7790 & 2.5014 & 3.0477 \\
TGARCH & 0.1888 & 0.2002 & \textbf{0.1007} & 0.1921 & 0.3886 & 0.4126 & 0.2343 & 0.3466 & 0.7845 & 0.8193 & 0.6661 & 1.0053 \\
CAViar &0.1335 & 0.1685 & 0.2312 & 0.3730 &0.1223 & 0.1568 & 0.2425 & 0.3176 & 0.2166 & 0.2814 & 0.4765 & 0.6493\\
CARE   &0.1224 & 0.1510 & 0.2105 & 0.3451 &0.1294 & 0.1559 & 0.2336 & 0.3165 & 0.1987 & 0.2778 & 0.4825 & 0.6701\\
QGARCH & \textbf{0.0575} & \textbf{0.0787} & 0.1244 & 0.2220 & 0.1215 & 0.1572 & 0.2734 & 0.3744 & 0.2492 & 0.3175 & 0.5513 & 0.7244\\
\hline
CALS-EL1 & 0.0837 & 0.1054 & 0.1146 & 0.1640 & \textbf{0.1080} & 0.1396 & 0.1246 & 0.2512 & 0.1817 & 0.2567 & 0.2434 & 0.4302\\
CALS-EL2 & 0.0829 & 0.1030 & 0.1141 & \textbf{0.1637} & 0.1100 & \textbf{0.1389} & \textbf{0.1171} & \textbf{0.2243} & \textbf{0.1759} & \textbf{0.2500} & \textbf{0.2425} & \textbf{0.4109}\\
\hline
\end{tabular}
\end{table}

\begin{table}[htbp]
\caption{Simulation results of conditional ES estimation}
\scriptsize
\begin{tabular}{lcccc|cccc|ccccc}
\hline
\multicolumn{13}{c}{ES Estiamtion}\\
\hline
\multirow{3}{*}{Method} & \multicolumn{4}{c}{Case 1}&\multicolumn{4}{c}{Case 2}&\multicolumn{4}{c}{Case 3}\\ \cline{2-5}\cline{6-9}\cline{10-13} & \multicolumn{2}{c}{normal} & \multicolumn{2}{c}{t(4)} & \multicolumn{2}{c}{normal} & \multicolumn{2}{c}{t(4)} & \multicolumn{2}{c}{normal} & \multicolumn{2}{c}{t(4)}\\ \cline{2-3}\cline{4-5}\cline{6-7}\cline{8-9}\cline{10-11}\cline{12-13}
& Bias & RMSE & Bias & RMSE & Bias & RMSE & Bias & RMSE & Bias & RMSE & Bias & RMSE  \\
\hline
GGARCH & 0.1660 & 0.2289 & 0.6934 & 0.7923 & 0.2918 & 0.3970 & 2.2530 & 2.5717 & 0.7457 & 0.9769 & 3.1369 & 3.8220 \\
TGARCH & 0.2837 & 0.3008 & 0.1513 & 0.2886 & 0.5838 & 0.6198 & 0.3520 & 0.5207 & 1.1786 & 1.2309 & 1.0007 & 1.5103 \\
CARE & 0.1289 & 0.1556 & 0.2411 & 0.3744 & 0.1321 & 0.1625 & 0.2488 & 0.3353 & 0.2219 & 0.2864 & 0.5223 & 0.6940 \\
QGARCH-EE & 0.1728 & 0.2482 & 0.3241 & 0.5308 & 0.2877 & 0.4496 & 0.5002 & 0.6491 & 0.4904 & 0.7593 & 0.9128 & 1.2804\\
QGARCH-EL & \textbf{0.0917} & 0.1183 & \textbf{0.1205} & 0.2207 & 0.1248 & 0.1609 & 0.1470 & 0.3208 & 0.2140 & \textbf{0.2553} & \textbf{0.2808} & 0.5314\\
\hline
CALS-EL1 & 0.1015 & 0.1162 & 0.1325 & 0.2027 & 0.1187 & 0.1425 & 0.1456 & 0.2977 & 0.1929 & 0.2689 & 0.2829 & 0.5339\\
CALS-EL2 & 0.1001 & \textbf{0.1136} & 0.1316 & \textbf{0.2020} & \textbf{0.1187} & \textbf{0.1418} & \textbf{0.1362} & \textbf{0.2909} & \textbf{0.1852} & 0.2640 & 0.2816 &\textbf{0.5181}\\
\hline
\end{tabular}
\end{table}

The simulation results show that the proposed method has good performance in general for both VaR and ES estimations. For VaR estimation, our method has a smaller bias and RMSE for most cases, except the QGARCH method has a smaller bias and MSE for case 1 with normal innovation, and the TGARCH method has a smaller bias for case 1 with t(4) innovation. For ES estimation, CALS-EL2 and QGARCH-EL are the most two competitive methods. QGARCH-EE and QGARCH-EL both estimate the volatility structure first and then estimate the tail-related risks from the standardised residuals, and they follow the same procedure as the proposed method. We can see that the QGARCH-EL method significantly outperforms QGARCH-EE, which indicates the empirical likelihood is helpful in processing the standardised residuals. Additionally, this result implies that empirical likelihood can be combined with other volatility estimation methods.

In addition to the overall comparison result, there are two noteworthy details. First, compared to the other methods, the performance of our method is not sensitively influenced by the choice of the innovation distribution. Take the VaR estimation as the example. When the innovation distribution is normal, the bias and RMSE of QGARCH and CALS-EL do not appear to be much different. Even in case 1, QGARCH has better performance than CALS-EL. While the data are generated from t(4) innovation, the advantage of CALS-EL begins to stand out. The bias and RMSE of CALS-EL increase by 27.3\% and 55.6\%, respectively, when the innovation distribution changes from normal to t(4). These increases are 120.9\% and 182.1\% for the QGARCH method, and they are even bigger for other methods. This phenomenon still exists when estimating the conditional ES. In summary, the CALS-EL method is more robust to the innovation distribution than are the other methods. That is theoretically reasonable because this method is distribution-assumption-free for innovation series. We do not need to worry too much about the estimation effect for heavy-tailed innovation cases because the necessary assumption for innovation distribution (assumption [C2]) is not strict, even for some heavy-tailed distributions.

Another highlight is the performance of CALS-EL when the stationary condition of GARCH model is not 'good' enough. As we mentioned before, the bigger $|\beta_1|$ and $|\beta_1|+|\gamma_1|$ are, the closer the series are to the boundary of stationary condition. In case 2 and case 3, the CALS-EL outperforms the other methods. In the other word, our method is more adjusted to the the cases when the GARCH series is nearly nonstationary.

The simulation results show that CALS-EL has a good estimation effect in general. The adaptability of our method for different innovation distributions and different stationary cases implies that it will have broad application for different conditional heteroscedastic time series.

\section{Empirical application to global financial indices}
In the empirical illustration part, we follow the idea of \cite{taylor_estimating_2008} to make day-ahead estimations of VaR and ES for different stock indices. We consider six global financial indices: the British FTSE100, the French CAC40, the German DAX30, the Hong Kong HSI, the Japanese Nikkei225 and the US S\&P500. We use 2000 log returns before Sep 31, 2017, for practical evaluation. For each method, 1000 days of data will be used as in-sample, and out-of-sample prediction of VaR and ES will be produced for the left 1000 days.

\begin{table}[!h]
\caption{Summary statistics of log returns from six stock indices.}
\centering
\begin{tabular}{lcccccc}
\toprule
Index & Mean & Min & Max & S. dev. & Skew. & Kurt.\\
\midrule
FTSE100 & $2.630\times10^{-4}$ & -0.0478 & 0.0514 & 0.0100 & -0.0736 & 5.767\\
CAC40 & $2.299\times10^{-4}$ & -0.0838 & 0.0922 & 0.0134 & -0.1213 & 6.824\\
DAX30 & $4.629\times10^{-4}$ & -0.0707 & 0.0609 & 0.0129 & -0.2807& 5.627\\
HSI & $2.162\times10^{-4}$ & -0.0602 & 0.0699 & 0.0118 & -0.2078 & 5.566\\
Nikkei225 & $3.411\times10^{-4}$ & -0.1115 & 0.0743 & 0.0136 & -0.4799 & 8.090\\
SP500 & $4.922\times10^{-4}$ & -0.0690 & 0.0546 & 0.0096 & -0.3089 & 7.871\\
\bottomrule
\end{tabular}
\end{table}

As in the simulation, CALS-EL uses a truncated ARCH series with $m=13$ and 19 equally spaced position quantiles ($\tau_k = 5k\%, k=1,\ldots,19$) to perform the preliminary volatility estimation. Eq. (\ref{eq52}) and Eq. (\ref{eq53}) are used for the out-of-sample conditional tail-related risk estimations. Figure \ref{FTSE100estimate} presents an example of estimation for the FTSE100 index. The benchmark methods are similar to those used in the simulation, but with minor adjustments. The QGARCH-EE method is removed since it is not competitive with QGARCH-EL. Moreover, considering the leverage effect, we add two asymmetric GARCH-based methods, including GJRGARCH-EE and GJRGARCH-EE. They use the GJRGARCH(1,1) of \cite{Glosten1993} to model the volatility and obtain the VaR and ES estimates from the standardised returns via empirical estimation and the extreme value approach, respectively.

\begin{figure}[!h]
\begin{center}
\includegraphics[width=\textwidth]{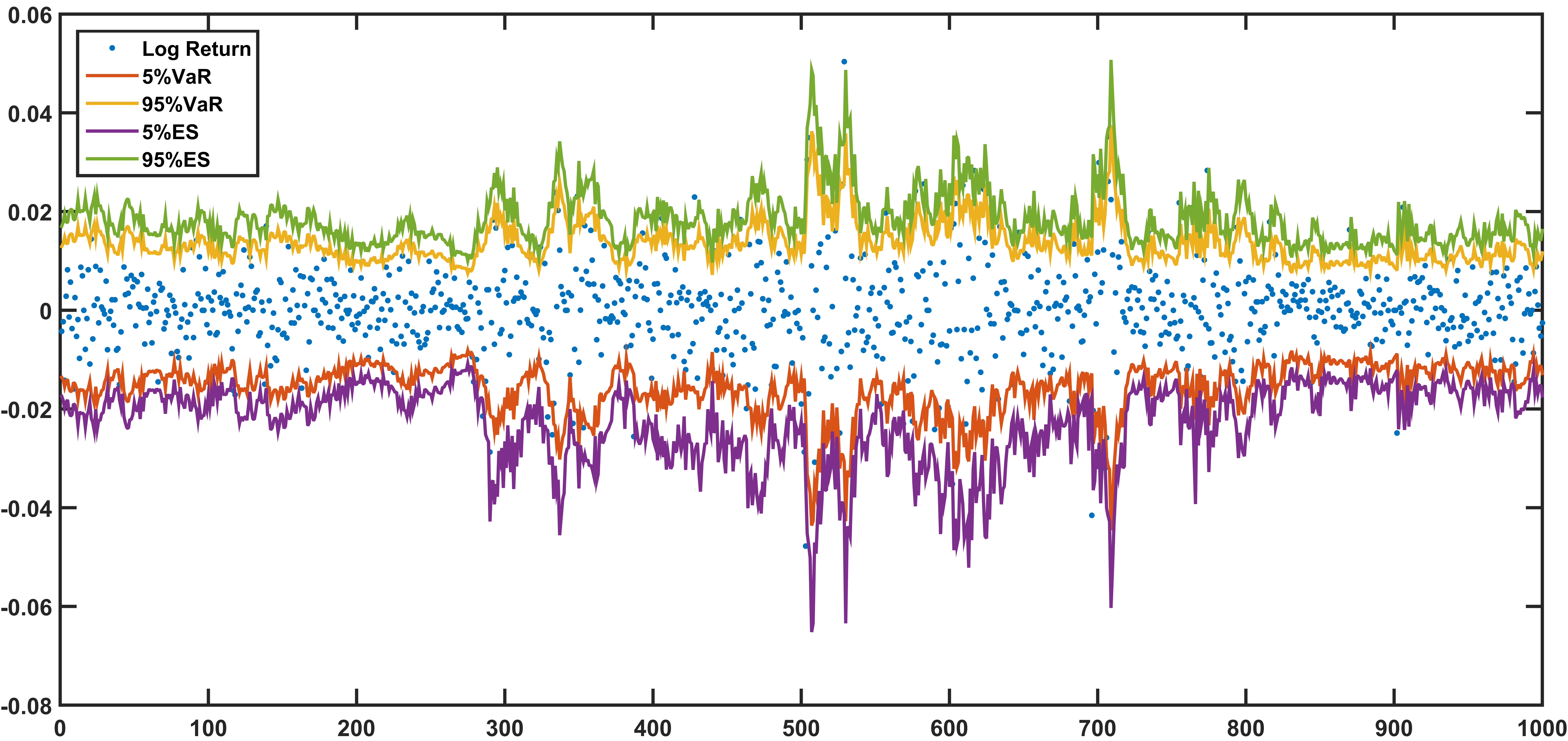}
\captionsetup{font={small}}
\caption{FTSE100 daily log-returns for the 1000 post-sample days with VaR and ES estimation using the CALS-EL method.}	
\label{FTSE100estimate}
\end{center}
\end{figure}

In addition to the estimations of conditional VaR and conditional ES, another value in which we are interested is the corresponding $\tau$ estimated via empirical likelihood. In fact, this quantity is time-variant since the distribution of innovation series (or the innovation that we observed, to be exact) changes over time for most financial time series. Take the result for the FTSE100 series as an example. Figure 3 shows the estimated $\tau$ series in processing the FTSE100 series for $\alpha=0.95$. As we can see, this series fluctuates with time, so if we use a fixed $\tau$ for the whole estimation procedure, it may cause a modelling error. The red line in the figure is the true value of the $\tau$ for normal distribution, and it is significantly above the $\tau$ that we select. Clearly, the normal distribution assumption is not a suitable choice here. If one must use a fixed $\tau$ value, a relatively reasonable choice is to use the corresponding $\tau$ of the distribution t(10) in the first three quarters of the window (the purple line) and to use the corresponding $\tau$ of distribution t(6) in the left quarter of the window (the yellow line). However, in practice, it is difficult to make a such predictive decision about the $\tau$ selection in advance. Regardless, this example has given us a lesson that it is not a good choice to select a fixed $\tau$ when estimating the conditional ES from expectile, which has not been noted previously.

\begin{figure}[!h]
\begin{center}
\includegraphics[width=\textwidth]{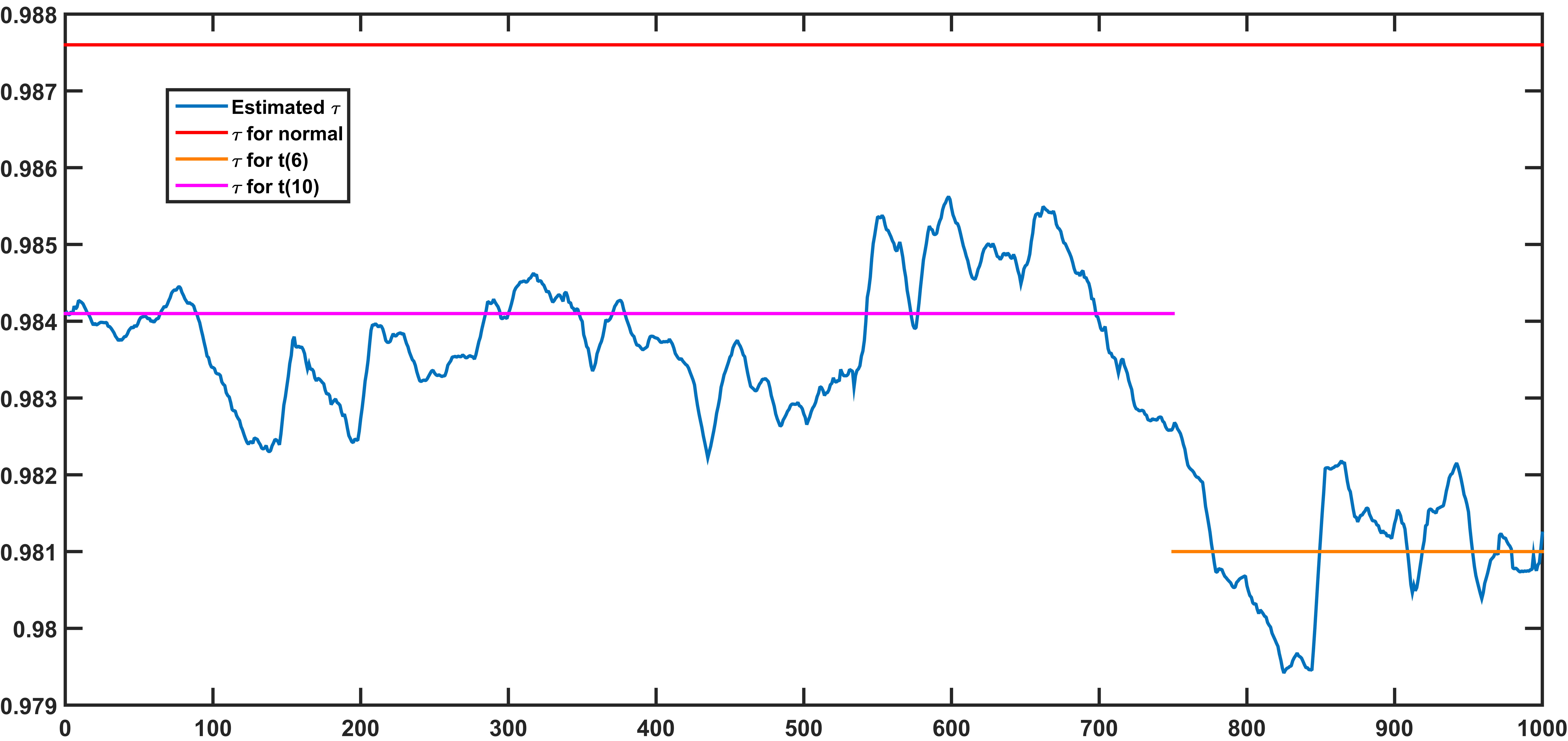}
\captionsetup{font={small}}
\caption{The value of $\tau$ estimated during the estimation procedure using CALS-EL for the FTSE100 index with $\alpha=0.95$.}	
\label{FTSE100estimate}
\end{center}
\end{figure}

Finally, we perform some backtesting to evaluate the results of empirical illustration. For conditional VaR estimation, we perform the Kupiec test \citep{Kupiec1995Techniques} and dynamic quantile (DQ) test \citep{engle_caviar_2004}, which both evaluate the estimation through the Hit variable defined as
\begin{eqnarray}
Hit_t=I(Y_t\geq\hat{Q}_\alpha(Y_t|\mathcal{F}_{t-1}))-\alpha.
\end{eqnarray}
The Kupiec test focuses on testing whether the $Hit_t$ has a zero mean or whether the percentage of real data exceeding the estimates equals $\alpha$. The DQ test proposed by \cite{engle_caviar_2004} involves the joint test of whether the hit variable has a Bernoulli distribution with $\alpha$ and is dependent on the lagged conditional VaR. We use the same lagged coefficient as in \cite{taylor_estimating_2008}, including four lags of hit variable and one lag of conational VaR estimate. Hence, the DQ test statistic is asymptotically distributed with $\chi^2(6)$.

For the conditional ES estimation, we use two evaluation tests. The first one is a bootstrap-based method proposed by \cite{mcneil_estimation_2000}, which tests whether the exceeding points of return have the same mean value as the conditional ES estimation. The other is the general conditional calibration (GCC) test recently proposed by \cite{nolde_elicitability_2017}, which is based on the joint elicitability of VaR and ES.

\begin{table}[!h]
\centering
\scriptsize
\newcommand{\tabincell}[2]{\begin{tabular}{@{}#1@{}}#2\end{tabular}}
\caption{Evaluation of conditional VaR and ES estimation for $\alpha=0.95$}
\begin{tabular}{cccccccccc}
\toprule
\multicolumn{2}{l}{ } & CAC40 & DAX30 & FTES100 & NIKKEI225 & SP500 & HSI &\tabincell{c}{Num of rejection\\ VaR test \\ at 5\% level} & \tabincell{c}{Num of rejection\\ ES test \\ at 5\% level}\\
\midrule
\multicolumn{10}{l}{\footnotesize\textbf{GGARCH}}\\
\multicolumn{2}{l}{Converge rate} & 0.032 & 0.045 & 0.034 & 0.036 & 0.029 & 0.040 & \multirow{5}{*}{\normalsize{6}} & \multirow{5}{*}{\normalsize{4}}\\
\cdashline{1-8}
\multirow{2}{*}{VaR Evaluation} & Kupiec & \textbf{0.005} & 0.544 & \textbf{0.014} & \textbf{0.033} &\textbf{0.001} & 0.133 \\
 & DQ &\textbf{0.007} & 0.894 & \textbf{0.002} & 0.163 & 0.051 & 0.685\\
\cdashline{1-8}
\multirow{2}{*}{ES Evaluation} & Bootstrap & \textbf{0.048} & 0.210 & 0.300 & 0.067 & 0.467 & 0.197\\
 & GCC & \textbf{0.005} & 0.430 & \textbf{0.023} & 0.074 & \textbf{0.000} & 0.446 \\
\midrule
\multicolumn{10}{l}{\footnotesize\textbf{TGARCH}}\\
\multicolumn{2}{l}{Converge rate} & 0.034 & 0.050 & 0.037 & 0.039 & 0.030 & 0.045 & \multirow{5}{*}{\normalsize{5}} & \multirow{5}{*}{\normalsize{7}}\\
\cdashline{1-8}
\multirow{2}{*}{VaR Evaluation} & Kupiec & \textbf{0.014} & 1.000 & \textbf{0.048} & 0.098 & \textbf{0.002} & 0.544 \\
 & DQ & 0.150 & 0.881 & \textbf{0.021} & 0.447 & \textbf{0.008} & 0.694\\
\cdashline{1-8}
\multirow{2}{*}{ES Evaluation} & Bootstrap & 0.278 & \textbf{0.002} & \textbf{0.000} & 0.219 & \textbf{0.000} & 0.078\\
 & GCC & \textbf{0.023} & \textbf{0.021} & \textbf{0.014} & 0.303 & \textbf{0.001} & 0.183 \\
 \midrule
\multicolumn{10}{l}{\footnotesize\textbf{GJRGARCH-EE}}\\
\multicolumn{2}{l}{Converge rate} & 0.043 & 0.044 & 0.046 & 0.041 & 0.036 & 0.041 & \multirow{5}{*}{\normalsize{1}} & \multirow{5}{*}{\normalsize{8}}\\
\cdashline{1-8}
\multirow{2}{*}{VaR Evaluation} & Kupiec & 0.299 & 0.375 & 0.577 & 0.178 &\textbf{0.033} & 0.178\\
 & DQ & 0.133 & 0.160 & 0.737 & 0.682 & 0.212 & 0.788\\
\cdashline{1-8}
\multirow{2}{*}{ES Evaluation} & Bootstrap & \textbf{0.002} & \textbf{0.008} & \textbf{0.008} & \textbf{0.025} & 0.075 & 0.116 \\
 & GCC & \textbf{0.005} & \textbf{0.013} & \textbf{0.015} & 0.052 & \textbf{0.007} & 0.253 \\
 \midrule
\multicolumn{10}{l}{\footnotesize\textbf{GJRGARCH-EV}}\\
\multicolumn{2}{l}{Converge rate} & 0.048 & 0.052 & 0.048 & 0.051 & 0.039 & 0.050 & \multirow{5}{*}{\normalsize{1}} & \multirow{5}{*}{\normalsize{1}}\\
\cdashline{1-8}
\multirow{2}{*}{VaR Evaluation} & Kupiec & 0.770 & 0.773 & 0.770 & 0.958 & 0.097 & 1.000 \\
 & DQ & 0.305 & \textbf{0.027} & 0.802 & 0.736 & 0.323 & 0.740\\
\cdashline{1-8}
\multirow{2}{*}{ES Evaluation} & Bootstrap & 0.261 & 0.532 & 0.543 & 0.328 & 0.340 & 0.138 \\
 & GCC & 0.308 & 0.730 & 0.643 & 0.702 & \textbf{0.043} & 0.267 \\
\midrule
\multicolumn{10}{l}{\footnotesize\textbf{Caviar}}\\
\multicolumn{2}{l}{Converge rate} & 0.046 & 0.046 & 0.046 & 0.048 & 0.039 & 0.046 & \multirow{5}{*}{\normalsize{0}} & \multirow{5}{*}{\normalsize{-}}\\
\cdashline{1-8}
\multirow{2}{*}{VaR Evaluation} & Kupiec & 0.557 & 0.557 & 0.557 & 0.770 & 0.097 & 0.557\\
 & DQ & 0.181 & 0.088 & 0.735 & 0.813 & 0.333 & 0.658\\
\cdashline{1-8}
\multirow{2}{*}{ES Evaluation} & Bootstrap & - & - & - & - & - & - \\
 & GCC & - & - & - & - & - & - \\
\midrule
\multicolumn{10}{l}{\footnotesize\textbf{CARE}}\\
\multicolumn{2}{l}{Converge rate} & 0.054 & 0.066 & 0.050 & 0.055 & 0.047 & 0.052 & \multirow{5}{*}{\normalsize{2}} & \multirow{5}{*}{\normalsize{0}}\\
\cdashline{1-8}
\multirow{2}{*}{VaR Evaluation} & Kupiec & 0.556 & \textbf{0.027} & 1.000 & 0.475 & 0.660 & 0.773 \\
 & DQ & 0.255 & \textbf{0.011} & 0.952 & 0.628 & 0.255 & 0.679\\
\cdashline{1-8}
\multirow{2}{*}{ES Evaluation} & Bootstrap & 0.292 & 0.235 & 0.295 & 0.310 & 0.726 & 0.112\\
 & GCC & 0.618 & 0.516 & 0.396 & 0.679 & 0.654 & 0.225 \\
\midrule
\multicolumn{10}{l}{\footnotesize\textbf{QGARCH(-EL)}}\\
\multicolumn{2}{l}{Converge rate} & 0.048 & 0.053 & 0.051 & 0.051 & 0.048 & 0.050 & \multirow{5}{*}{\normalsize{1}} & \multirow{5}{*}{\normalsize{2}}\\
\cdashline{1-8}
\multirow{2}{*}{VaR Evaluation} & Kupiec & 0.770 & 0.666 & 0.885 & 0.885 & 0.770 & 1.000 \\
 & DQ & 0.289 & \textbf{0.012} & 0.812 & 0.674 & 0.236 & 0.737 \\
\cdashline{1-8}
\multirow{2}{*}{ES Evaluation} & Bootstrap & \textbf{0.047} & 0.119 & 0.154 & 0.066 & 0.753 & 0.532 \\
 & GCC & \textbf{0.044} & 0.131 & 0.174 & 0.133 & 0.610 & 0.927 \\
\midrule
\multicolumn{10}{l}{\footnotesize\textbf{CALS-EL}}\\
\multicolumn{2}{l}{Converge rate} & 0.039 & 0.046 & 0.044 & 0.052 & 0.044 & 0.038  & \multirow{5}{*}{\normalsize{1}} & \multirow{5}{*}{\normalsize{0}}\\
\cdashline{1-8}
\multirow{2}{*}{VaR Evaluation} & Kupiec & 0.097 & 0.557 & 0.375 & 0.773 & 0.375 & 0.070 \\
 & DQ & 0.137 & 0.399 & 0.511 & 0.708 & 0.673 & \textbf{0.034}\\
\cdashline{1-8}
\multirow{2}{*}{ES Evaluation} & Bootstrap & 0.682 & 0.086 & 0.541 & 0.126 & 0.175 & 0.483\\
 & GCC & 0.515 & 0.146 & 0.516 & 0.053 & 0.333 & 0.178\\
\bottomrule
\end{tabular}
\end{table}

\begin{table}[!h]
\centering
\scriptsize
\newcommand{\tabincell}[2]{\begin{tabular}{@{}#1@{}}#2\end{tabular}}
\caption{Evaluation of conditional VaR and ES estimation for $\alpha=0.99$}
\begin{tabular}{cccccccccc}
\toprule
\multicolumn{2}{l}{ } & CAC40 & DAX30 & FTES100 & NIKKEI225 & SP500 & HSI &\tabincell{c}{Num of rejection\\ VaR test \\ at 5\% level} & \tabincell{c}{Num of rejection\\ ES test \\ at 5\% level}\\
\midrule
\multicolumn{10}{l}{\footnotesize\textbf{GGARCH}}\\
\multicolumn{2}{l}{Converge rate} & 0.013 & 0.017 & 0.015 & 0.022 & 0.017 & 0.015 & \multirow{5}{*}{\normalsize{5}} & \multirow{5}{*}{\normalsize{11}}\\
\cdashline{1-8}
\multirow{2}{*}{VaR Evaluation} & Kupiec & 0.362 & \textbf{0.043} & 0.139 & \textbf{0.000} &\textbf{0.043} & 0.139 \\
 & DQ & 0.924 & 0.352 & 0.056& \textbf{0.008} & \textbf{0.000} & 0.098\\
\cdashline{1-8}
\multirow{2}{*}{ES Evaluation} & Bootstrap & \textbf{0.000} & \textbf{0.000} & \textbf{0.027} & \textbf{0.004} & \textbf{0.000} & \textbf{0.000}\\
 & GCC & \textbf{0.023} & \textbf{0.025} & 0.098 & \textbf{0.042} & \textbf{0.033} & \textbf{0.030} \\
\midrule
\multicolumn{10}{l}{\footnotesize\textbf{TGARCH}}\\
\multicolumn{2}{l}{Converge rate} & 0.012 & 0.014 & 0.012 & 0.018 & 0.014 & 0.014 & \multirow{5}{*}{\normalsize{3}} & \multirow{5}{*}{\normalsize{4}}\\
\cdashline{1-8}
\multirow{2}{*}{VaR Evaluation} & Kupiec & 0.538 & 0.231 & 0.380 & \textbf{0.022} & 0.231 & 0.231 \\
 & DQ & 0.976 & 0.925 & \textbf{0.016} & 0.100 & \textbf{0.000} & 0.069 \\
\cdashline{1-8}
\multirow{2}{*}{ES Evaluation} & Bootstrap & 0.401 & 0.437 & \textbf{0.034} & \textbf{0.016} & \textbf{0.000} & 0.342\\
 & GCC & 0.883 & 0.964 & 0.149 & 0.097 & \textbf{0.029} & 0.819 \\
 \midrule
\multicolumn{10}{l}{\footnotesize\textbf{GJRGARCH-EE}}\\
\multicolumn{2}{l}{Converge rate} & 0.022 & 0.011 & 0.019 & 0.012 & 0.008 & 0.011 & \multirow{5}{*}{\normalsize{4}} & \multirow{5}{*}{\normalsize{6}}\\
\cdashline{1-8}
\multirow{2}{*}{VaR Evaluation} & Kupiec & \textbf{0.001} & 0.754 & \textbf{0.011} & 0.534 & 0.514 & 0.754\\
 & DQ & \textbf{0.000} & 0.715 & \textbf{0.028} & 0.921 & 0.954 & 0.201\\
\cdashline{1-8}
\multirow{2}{*}{ES Evaluation} & Bootstrap & \textbf{0.025} & \textbf{0.004} & 0.087 & \textbf{0.007} & \textbf{0.021} & \textbf{0.044} \\
 & GCC & 0.121 & 0.057 & 0.194 & \textbf{0.049} & 0.125 & 0.165 \\
 \midrule
\multicolumn{10}{l}{\footnotesize\textbf{GJRGARCH-EV}}\\
\multicolumn{2}{l}{Converge rate} & 0.010 & 0.012 & 0.007 & 0.015 & 0.008 & 0.007 & \multirow{5}{*}{\normalsize{0}} & \multirow{5}{*}{\normalsize{2}}\\
\cdashline{1-8}
\multirow{2}{*}{VaR Evaluation} & Kupiec & 1.000 & 0.538 & 0.314 & 0.139 & 0.510 & 0.314 \\
 & DQ & 0.987 & 0.919 & 0.967 & 0.802 & 0.911 & 0.984\\
\cdashline{1-8}
\multirow{2}{*}{ES Evaluation} & Bootstrap & \textbf{0.049} & 0.292 & 0.412 & \textbf{0.048} & 0.318 & 0.410 \\
 & GCC & 0.169 & 0.614 & 0.899 & 0.142 & 0.589 & 0.792 \\
\midrule
\multicolumn{10}{l}{\footnotesize\textbf{Caviar}}\\
\multicolumn{2}{l}{Converge rate} & 0.012 & 0.013 & 0.008 & 0.017 & 0.009 & 0.010 & \multirow{5}{*}{\normalsize{1}} & \multirow{5}{*}{\normalsize{-}}\\
\cdashline{1-8}
\multirow{2}{*}{VaR Evaluation} & Kupiec & 0.538 & 0.362 & 0.510 & \textbf{0.043} & 0.746 & 1.000\\
 & DQ & 0.949 & 0.903 & 0.979 & 0.473 & 0.857 & 0.989 \\
\cdashline{1-8}
\multirow{2}{*}{ES Evaluation} & Bootstrap & - & - & - & - & - & - \\
 & GCC & - & - & - & - & - & - \\
\midrule
\multicolumn{10}{l}{\footnotesize\textbf{CARE}}\\
\multicolumn{2}{l}{Converge rate} & 0.012 & 0.012 & 0.011 & 0.017 & 0.012 & 0.013 & \multirow{5}{*}{\normalsize{2}} & \multirow{5}{*}{\normalsize{2}}\\
\cdashline{1-8}
\multirow{2}{*}{VaR Evaluation} & Kupiec & 0.538 & 0.538 & 0.754 & \textbf{0.043} & 0.538 & 0.362 \\
 & DQ & 0.998 & 0.993 & 0.181 & 0.111 & \textbf{0.000} & 0.056\\
\cdashline{1-8}
\multirow{2}{*}{ES Evaluation} & Bootstrap & 0.464 & 0.182 & 0.212 & \textbf{0.049} & \textbf{0.006} & 0.352\\
 & GCC & 0.996 & 0.392 & 0.407 & 0.165 & 0.100 & 0.787 \\
\midrule
\multicolumn{10}{l}{\footnotesize\textbf{QGARCH(-EL)}}\\
\multicolumn{2}{l}{Converge rate} & 0.013 & 0.014 & 0.009 & 0.019 & 0.009 & 0.014 & \multirow{5}{*}{\normalsize{1}} & \multirow{5}{*}{\normalsize{3}}\\
\cdashline{1-8}
\multirow{2}{*}{VaR Evaluation} & Kupiec & 0.362 & 0.231 & 0.746 & \textbf{0.011} & 0.746 & 0.231 \\
 & DQ & 0.873 & 0.772 & 0.995 & 0.160 & 0.864 & 0.081 \\
\cdashline{1-8}
\multirow{2}{*}{ES Evaluation} & Bootstrap & \textbf{0.012} & \textbf{0.003} & 0.279 & 0.422 & 0.390 & 0.136 \\
 & GCC & 0.059 & \textbf{0.042} & 0.621 & 0.924 & 0.786 & 0.292 \\
\midrule
\multicolumn{10}{l}{\footnotesize\textbf{CALS-EL}}\\
\multicolumn{2}{l}{Converge rate} & 0.012 & 0.009 & 0.009 & 0.013 & 0.010 & 0.007  & \multirow{5}{*}{\normalsize{1}} & \multirow{5}{*}{\normalsize{3}}\\
\cdashline{1-8}
\multirow{2}{*}{VaR Evaluation} & Kupiec & 0.538 & 0.746 & 0.746 & 0.362 & 1.000 & 0.314 \\
 & DQ & 0.733 & 0.998 & 0.781 & 0.294 & \textbf{0.011} & 0.067\\
\cdashline{1-8}
\multirow{2}{*}{ES Evaluation} & Bootstrap & 0.261 & 0.416 & \textbf{0.001} & 0.146 & \textbf{0.002} & 0.164\\
 & GCC & 0.350 & 0.956 & \textbf{0.037} & 0.341 & 0.058 & 0.258\\
\bottomrule
\end{tabular}
\end{table}

Table 4 and Table 5 list the 95\% and 99\% evaluation results, including the coverage rates of the estimation (the frequency of the Hit series is equal to 1) and p-values of different tests. The final two columns present the counts for the number of series for which the null is rejected at the 5\% significance level (the former for the VaR test and the latter for the ES test), which are the main evaluation bases of comparison. We can see that these improved methods, including GJRGARCH-EV, QGARCH, Caviar, CARE and CALS-EL, perform better than does the fitter GARCH model with a certain distribution assumption. GJRGARCH-EV, QGARCH and CaViar put particular emphasis on the conditional VaR estimation by constructing a conditional VaR specification, so they have a better effect in this respect. The performance of CARE and CALS-EL (both of which are based on expectile) in conditional VaR estimation is also acceptable, although they do not fit the conditional VaR specification directly. The speciality of CARE and CALS-EL is conditional ES estimation. Predicting the conditional ES through expectiles has greater advantages than does using the GARCH model directly, as demonstrated by the evaluation results. In summary, the CALS-EL is competitive for conditional tail-related risk estimation, including VaR and ES.

\section{Conclusion and further discussion}
In this paper, we proposed a method for estimating the conditional VaR and conditional ES for conditional heteroscedastic time series. The proposed method uses CALS for volatility estimation and empirical likelihood for determining the $\tau$ for a fixed $\alpha$. The estimation procedure first estimates the volatility and then obtains the tail-related risks from the innovation series based on the relations among the VaR, ES and expectile. Asymptotic properties of the method are given in the paper. We also evaluate the performance of the method using both simulation study and an empirical application.

The main contributions of our work can be summarised as follows. First, we extend the method of \cite{xiao_conditional_2009} for volatility estimation in CALS, which avoids complex matrix decomposition calculations. Second, we use the empirical likelihood method to determine the $\tau$, which is an important contribution of our method. The empirical likelihood method provides us with a distribution-assumption-free and flexible data-driven way to determine the mapping $\tau=h_\varepsilon(\alpha)$, which fundamentally connects the VaR, ES and expectile.

Although we apply our proposed method to the GARCH model, we should note that the method is not limited to solving the linear GARCH model. The proposed method could be also applied in other cases by using different optimisation functions in Eq. (\ref{hatp}) to fit other conditional heteroscedastic series. Furthermore, the empirical likelihood estimation procedure is relatively independent of the volatility estimation process, only requiring an advisable approximation to the original innovation series. Hence, it can be combined with other volatility estimation procedures. This characteristic makes the proposed method applicable to a wide class of estimation scenarios.

For the tail-related risk, ES is more difficult to predict and backtest since it is not elicitable as VaR is. Expectile provides more potential to analyse ES due to its relation with VaR and ES. Finding the mapping $\tau=h(\alpha)$ from the quantile significance level $\alpha$ to the expectile significance level $\tau$ is an important issue when using expectile to analyse VaR and ES. Hence, the method to determine the value of $\tau$ deserves more attention. In an empirical application, Figure 3 indicates that this mapping shares time-varying features, and the empirical likelihood provides us with data-driven inspiration to address this issue.

\section*{Acknowledgements}
The authors thank James Taylor for informing us some details about the CARE model. Yi Zhang thanks support from the Zhejiang Provincial Science Foundation (No: LY18A010005, LY17A010016), Zhejiang University Education Foundation ZJU-Stanford Collaboration Fund, and the Fundamental Research Funds for the Central Universities. Sheng Wu thanks support from the MOE Project of Humanities and Social Sciences (No. 17YJA910003).

\appendix
\newtheorem{lemapp}{\emph{Lemma}}
\newtheorem{defnapp}[lemapp]{\emph{Definition}}
\section{Proof}
\subsection{Proof of Theorem \ref{thm1}}
We use the conclusion in \cite{RePEc:eee:ecochp:4-45} to prove the consistence of $\tilde{\theta}$ in Theorem \ref{thm1}. Let us first give some definitions and lemmas in \cite{RePEc:eee:ecochp:4-45}.
\begin{defnapp}Let $\Theta\subset \mathbb{R}^p$, $\{w_t: t=1,2,\ldots\}$ be a sequence of random vectors with $w_t\in\mathcal{W}_t$, $t=1,2,\ldots$, and $\{\rho_t: t=1,2,\ldots\}$ be a sequence of real-valued functions, where $\mathcal{W}_t$ is the sample space of $\{w_t: t=1,2,\ldots\}$. Assume that
\begin{enumerate}[(a)]
\item $\Theta$ is compact;
\item $\rho_t$ satisfies the standard measurability and continuity on $\mathcal{W}_t\times\Theta, t=1,2,\ldots$;
\item $\mathbf{E}[|\rho_t(w_t,\theta)|]<\infty$ for all $\theta\in\Theta$, $t=1,2,\ldots$;
\item $\lim_{T\rightarrow\infty}n^{-1}\sum_{t=1}^{n}\mathbf{E}[\rho_t(w_t,\theta)]$ exists for all $\theta\in\Theta$;
\item $\max_{\theta\in\Theta}|n^{-1}\sum_{t=1}^{n}\rho_t(w_t,\theta)-\mathbf{E}[\rho_t(w_t,\theta)]|\stackrel{p}\rightarrow0$.
\end{enumerate}
Then $\{\rho_t(w_t,\theta)\}$ is said to satisfiers the uniform weak law of large number (denoted by UWLLN) on $\Theta$.
\end{defnapp}

The sufficient conditions of UWLLN for the stationary ergodic case is given by the following lemma in \cite{rao_relations_1962}.

\begin{lemapp}\label{lemmaapp2}
Let $\Theta\subset\mathbb{P}^p$, $\{w_t: t=1,2,\ldots\}$ be a sequence of stationary and ergodic $M\times1$ random vectors, and $\rho:\mathcal{W}\times\Theta\rightarrow\mathbb{R}$ be a real-valued function. With the following conditions, $\{\rho(w_t,\theta)\}$ satisfiers the UWLLN on $\Theta$.
\begin{enumerate}[(a)]
\item $\Theta$ is compact;
\item $\rho$ satisfies the standard measurability and continuity on $\mathcal{W}\times\Theta$;
\item for some function $b:\mathcal{W}\rightarrow\mathbb{R}$ with $\mathbf{E}[b(w_t)]<\infty$, and $|\rho(w,\theta)|\leq b(w)$ for all $\theta\in\Theta$.
\end{enumerate}
\end{lemapp}

These conditions are used to check that the optimisation function of the CALS satisfies the UWLLN on $\Theta$. Before that, we give a lemma belonging to \citet{RePEc:eee:ecochp:4-45}, from which the consistency of $\tilde{\theta}$ is derived.
\begin{lemapp}\label{lemmaapp3}
Let $\Theta\subset\mathbb{R}^p$, $\{w_t: t=1,2,\ldots\}$ be a sequence of random vectors and $\{\rho_t:\mathcal{W}_t\times\Theta\rightarrow\mathbb{R}, t=1,2,\ldots\}$ be the sequence of objective functions. Assume that
\begin{enumerate}[(a)]
\item $\Theta$ is compact;
\item $\rho_t$ satisfies the standard measurability and continuity on $\mathcal{W}_t\times\Theta, t=1,2,\ldots$;
\item $\{\rho_t(w_t,\theta), t=1,2,\ldots\}$ satisfies the UWLLN on $\Theta$;
\item On $\Theta$, there exists an unique minimizer $\theta_0$ of the function $$\bar{\rho}(\theta)\triangleq\lim_{n\rightarrow\infty}n^{-1}\sum_{t=1}^{n}\mathbf{E}[\rho_t(w_t,\theta)].$$
\end{enumerate}
Then there exists a random vector, $\tilde{\theta}$, which minimizes $\sum_{t=1}^{n}\rho(w_t,\theta)$ and satisfies $\tilde{\theta}\stackrel{p}\rightarrow\theta_0$, as $n\rightarrow\infty$.
\end{lemapp}

\begin{proof}[Proof of Theorem \ref{thm1}]
It is obvious that the conditions (a) and (b) in Lemma \ref{lemmaapp2} and Lemma \ref{lemmaapp3} are hold under assumptions [A1] and [A3]. Hence, we only need to check the condition (c) in Lemma \ref{lemmaapp2}, and then we can learn that the optimisation function of CALS satisfies the UWLLN on $\Theta$.

For notational convenience, we denote the loss function $\rho_{\tau_k}$ by $\rho_k$ in the following proof. We rewrite the loss function of CALS as $$\rho(w_t,\theta)=\sum_{k=1}^{K}\rho_k(w_t,\theta)=\sum_{k=1}^{K}\rho_k^s(w_t,\theta_k),$$ where
$\theta_k=(u_k a_0,u_k a_1,\ldots,u_k a_m)^T$ is reconstructed parameters from $\theta$, and $\rho_k^s(w_t,\theta_k)$ is the single asymmetric least squares loss function with parameters $\theta_k$, for $k=1,\ldots,K$. It has been proved by \cite{newey_asymmetric_1987} that for a single asymmetric least squares objective function $\rho_k(w_t,\theta_k)$, there exist constants $d_k$ and $d'_k$ such that $|\rho_k(w_t,\theta_k)|\leq|w_t|^2(d_k+d'_k|\theta_k|^2)$. Since $|\theta_k|^2\leq|\theta|^4$ for $k=1,\ldots,K$, we have $$|\rho(w_t,\theta)|\leq|w_t|^2(\sum_{k=1}^{K}d_k+\sum_{k=1}^{K}d'_k|\theta|^4).$$
Therefore, there exists a constant $M$ such that $|\rho(w_t,\theta)|\leq|w_t|^2M$ since $\Theta$ is compact, and it is easy to show that $\mathbf{E}[|w_t|^2M]<\infty$.

Now we have proved that the optimisation function of CALS satisfies the UWLLN on $\Theta$. In addition, \citet{newey_asymmetric_1987} proves that $\lim_{n\rightarrow\infty}n^{-1}\sum_{t=1}^{n}\mathbf{E}[\rho_k(w_t,\theta_k)]$ has a unique minimizer $\theta_{k0}$ for $k=1,\ldots,K$. Hence, $\lim_{n\rightarrow\infty}n^{-1}\sum_{t=1}^{n}\mathbf{E}[\rho(w_t,\theta)]=\lim_{n\rightarrow\infty}n^{-1}\sum_{t=1}^{n}\sum_{k=1}^{K}\mathbf{E}[\rho_k(w_t,\theta_k)]$ also has a unique minimizer $\theta_0$ when fix $a_0=1$, and the condition (d) in Lemma \ref{lemmaapp3} is hold. So far, the proof is completed with verifying all conditions in Lemma \ref{lemmaapp3}.
\end{proof}

\subsection{Proof of Theorem \ref{thm2}}
The proof for asymptotic normality of the CALS estimation follows the Theorem 3 in \cite{huber1967} and is similar to the proof of Theorem 3 in \cite{newey_asymmetric_1987}.
\begin{proof}
According to the conclusion in \cite{newey_asymmetric_1987} that $\mathbf{E}[\rho_k(w_t,\theta_k)]$ is twice continuously differentiable in $\theta_k$ for $k=1,\ldots,K$, we can learn that $\mathbf{E}[\rho(w_t,\theta)]$ is also twice continuously differentiable in $\theta$ with
\begin{eqnarray}
\zeta(\theta)=\partial\mathbf{E}[\rho(w_t,\theta)]/\partial\theta=\mathbf{E}[\psi(w_t,\theta)],
\end{eqnarray}
and
\begin{eqnarray}
\partial\zeta(\theta)/\partial\theta=\partial^{2}\mathbf{E}[\rho(w_t,\theta)]/\partial\theta\partial\theta'=\mathbf{E}[\partial\psi(w_t,\theta)/\partial\theta'],
\end{eqnarray}
where $\psi(w_t,\theta)\triangleq\partial\rho(w_t,\theta)/\partial\theta$. Since there are two parts of parameters in $\theta$, the derivative of the optimisation function has a more complex form. For notational convenience, we use the following notation $$v_k(w_t,\theta)\triangleq Y_t-u_k(a_0+a_1|Y_{t-1}|+\ldots+a_m|Y_{t-m}|)=Y_t-u_k\eta^Tx_{t,(m)},$$ $$\varphi_k(w_t,\theta)\triangleq|\tau_k-\mathbf{1}_{v_k(w_t,\theta)\leq0}|,$$ for $k=1,\ldots,K$. Then, $\psi(w_t,\theta)$ can be represented as
\begin{eqnarray*}
\psi(w_t,\theta)& = &\sum_{k=1}^{K}\partial\rho_{k}(Y_t-\mu_k\eta^{T}x_{t,(m)})/\partial\theta =\sum_{k=1}^{K}\left[\begin{array}{c}\partial\rho_{k}(Y_t-\mu_k\eta^{T}x_{t,(m)})/\partial\vartheta\\\partial\rho_{k}(Y_t-\mu_k\eta^{T}x_{t,(m)})/\partial\eta\end{array}\right]\\
& = &\sum_{k=1}^{K} \left[\begin{array}{c} A_k \\ B_k \end{array}\right],\\
& \triangleq &\sum_{k=1}^{K} \left[\begin{array}{c}0\\\vdots\\-2\varphi_k(w_t,\theta)v_k(w_t,\theta)\eta^Tx_{t,(m)}\\\vdots\\0\\-2\varphi_k(w_t,\theta)v_k(w_t,\theta)u_kx_{t,(m)}\end{array}\right],
\end{eqnarray*}
where $A_k=[0,\ldots,-2\varphi_k(w_t,\theta)v_k(w_t,\theta)\eta^Tx_{t,(m)},\ldots,0]^T$ has dimension $K\times1$ and a nonzero element in $k$-th line and $B_k=-2\varphi_k(w_t,\theta)v_k(w_t,\theta)u_kx_{t,(m)}$, for $k=1,\ldots,K$.
\small\begin{eqnarray*}
&&\partial\psi(w_t,\theta)/\partial\theta'=\sum_{k=1}^{K}\left[\begin{array}{cc}\partial^2\rho_k(w_t,\theta)/\partial\vartheta\partial\vartheta' & \partial^2\rho_k(w_t,\theta)/\partial\vartheta\partial\eta'\\ \partial^2\rho_k(w_t,\theta)/\partial\eta\partial\vartheta' & \partial^2\rho_k(w_t,\theta)/\partial\eta\partial\eta'\end{array}\right]\\&=&\sum_{k=1}^{K}\left[\begin{array}{cc}{\left(\begin{array}{ccccc}0& & & & \\ &\ddots& & & \\ & &2\varphi_k(w_t,\theta)\eta^Tx_{t,(m)}x_{t,(m)}^T\eta& & \\ & & & \ddots& \\ & & & &0\end{array}\right)} & {\left(\begin{array}{c}0\\\vdots\\-2\varphi_k(w_t,\theta)(Y_t-2u_k\eta^Tx_{t,(m)})x_{t,(m)}^T\\\vdots\\0\end{array}\right)}\\ {\left(\begin{array}{c}0\\\vdots\\-2\varphi_k(w_t,\theta)(Y_t-2u_k\eta^Tx_{t,(m)})x_{t,(m)}^T\\\vdots\\0\end{array}\right)}^T & 2\varphi_k(w_t,\theta)u_k^2x_{t,(m)}x_{t,(m)}^T\end{array}\right].
\end{eqnarray*}
\normalsize
It is obvious that $\zeta(\theta_0)=0$ and Assumptions [A1]-[A3] are sufficient for [N1],[N2],[N4] in Theorem 3 of \cite{huber1967}. To prove the condition [N3] in \cite{huber1967}, we need to show that $\Sigma=\mathbf{E}[\partial\psi(w_t,\theta_0)/\partial\theta']$ is not singular.

Since $\mathbf{E}[\varphi_k(w_t,\theta)(Y_t-u_k\eta^Tx_{t,(m)})]_{|\theta=\theta_0}=0$ for $k=1,\ldots,K$, we have
\small\begin{eqnarray*}
&&\Sigma=\mathbf{E}[\partial\psi(w_t,\theta_0)/\partial\theta']\\
&=&\sum_{k=1}^{K}\mathbf{E}\left[\begin{array}{cc}{\left(\begin{array}{ccccc}0& & & & \\ &\ddots& & & \\ & &2\varphi_k(w_t,\theta)\eta^Tx_{t,(m)}x_{t,(m)}^T\eta& & \\ & & & \ddots& \\ & & & &0\end{array}\right)} & {\left(\begin{array}{c}0\\\vdots\\-2\varphi_k(w_t,\theta)u_k\eta^Tx_{t,(m)}x_{t,(m)}^T\\\vdots\\0\end{array}\right)}\\ {\left(\begin{array}{c}0\\\vdots\\-2\varphi_k(w_t,\theta)u_k\eta^Tx_{t,(m)}x_{t,(m)}^T\\\vdots\\0\end{array}\right)}^T & 2\varphi_k(w_t,\theta)u_k^2x_{t,(m)}x_{t,(m)}^T\end{array}\right]_{\theta=\theta_0}\\
&=&\sum_{k=1}^{K}\mathbf{E}[2\varphi_k(w_t,\theta)\xi_k(w_t,\theta_0)\xi_k^T(w_t,\theta_0)],
\end{eqnarray*}
\normalsize
where \begin{eqnarray*}\xi_k(w_t,\theta)=\left[\begin{array}{c}0\\\vdots\\\eta^Tx_{t,(m)}\\\vdots\\0\\-u_kx_{t,(m)}\end{array}\right].\end{eqnarray*}
Hence, $\Sigma$ is nonnegative definite, and it is not trivial to show that $\Sigma$ is not singular with Assumption [A4]. According to the nonsingularity of $\Sigma$, there are positive constants $\kappa$ and $\kappa'$ such that
\begin{eqnarray}
|\theta-\theta_0|<\kappa\Rightarrow|\zeta(\theta)|>\kappa'|\theta-\theta_0|.
\end{eqnarray}
Let $U(w_t,\theta,\kappa)=\sup_{|\theta'-\theta|<\kappa}|\psi(w_t,\theta')-\psi(w_t,\theta)|$. Since $\Theta$ is compact and $$|\varphi_k(w_t,\theta')v_k(w_t,\theta')-\varphi_k(w_t,\theta)v_k(w_t,\theta)|\leq|v_k(w_t,\theta')-v_k(w_t,\theta)|<|x_{t,(m)}||\theta'-\theta|,$$
for $k=1,\ldots,K$, we have
\begin{eqnarray}
U(w_t,\theta,\kappa)<2K|x_{t,(m)}|^2|\theta|^2|\theta'-\theta|\leq2K|x_{t,(m)}|^2d^2\kappa.
\end{eqnarray}
Then, according to Assumption [A3], we have
\begin{eqnarray}
\mathbf{E}[U(w_t,\theta,\kappa)]<2Kd^2M_2\kappa,
\end{eqnarray}
\begin{eqnarray}
\mathbf{E}[U^2(w_t,\theta,\kappa)]<4K^2d^4M_4\kappa^2,
\end{eqnarray}
where $M_2=\int|x|^2g(y_t|x_t)h(x_t)d\upsilon_w$ and $M_4=\int|x|^4g(y_t|x_t)h(x_t)d\upsilon_w$. Now we have verified the condition [N3] in Theorem 3 of \cite{huber1967}, so we have
\begin{eqnarray}
\sqrt{n}\zeta(\tilde{\theta})+\frac{1}{\sqrt{n}}\sum_{t=1}^{n}\psi(w_t,\theta_0)=o_p(1).
\end{eqnarray}
Through Taylor expansion of $\sqrt{n}\zeta(\tilde{\theta})$ at $\theta_0$, we have
\begin{eqnarray}
\sqrt{n}(\tilde{\theta}-\theta_0)=-(\partial\zeta(\ddot{\theta})/\partial\theta)^{-1}\frac{1}{\sqrt{n}}\sum_{t=1}^{n}\psi(w_t,\theta_0)+o_p(1),
\end{eqnarray}
where $\ddot{\theta}$ is a ponit between $\tilde{\theta}$ and $\theta_0$, which converges to $\theta_0$ by the consistency of $\tilde{\theta}$. And by the continuity of $\partial\zeta(\theta)/\partial\theta$, we can learn that $\partial\zeta(\ddot{\theta})/\partial\theta$ also converges to $\Sigma$ in probability. Meanwhile, the following central limit theorem for stationary sequence $\{\psi(w_t,\theta_0)\}_{t>0}$ can be obtained from Assumptions [A1]-[A2]
\begin{eqnarray}
\frac{1}{\sqrt{n}}\sum_{t=1}^{n}\psi(w_t,\theta_0)\stackrel{d}\rightarrow N(\mathbf{0},\Omega),
\end{eqnarray}
as $n\rightarrow\infty$.
\end{proof}

\subsection{Proof of Corollary \ref{cor1} and Corollary \ref{cor2}}
It is obvious that Corollary \ref{cor1} holds because $\eta$ is a part of $\theta$. The proof of Corollary \ref{cor2} is also trivial.

\subsection{Proof of Theorem \ref{thm3}}
This proof is similar to the proof of Theorem 3 in \cite{xiao_conditional_2009}. Some related conclusion will also be used here.
\begin{proof}
Since the optimisation function (\ref{hatp}) is a least square loss which can be rewritten as Eq. (\ref{ols}), we write its first order derivative as
$$H_n(\phi,\eta)=\frac{1}{n}\sum_t(\sigma_{t}(\eta)-\phi^Tz_t(\eta))z_t(\eta),$$. Let
$$H(\phi,\eta)=\mathbf{E}\left[(\sigma_{t}(\eta)-\phi^Tz_t(\eta))z_t(\eta)\right],$$
$$\Gamma_1(\phi,\eta)=\frac{\partial H(\phi,\eta)}{\partial \phi}=\mathbf{E}[z_t(\eta)z_t^T(\eta)],$$
then
$$\Gamma_{10}\triangleq\Gamma_1(\phi_0,\eta_0)=\mathbf{E}[z_t(\eta_0)z_t^T(\eta_0)]=\mathbf{E}[z_tz_t^T]+o_p(n^{-1}).$$
Thus, $\Gamma_{10}$ is nonsingular under Assumption [A4], and there exists a constant $C$ such that $C\parallel\hat{\phi}-\phi_0\parallel\leq \parallel H(\hat{\phi},\eta_0)\parallel$. To verify the $\sqrt{n}$ consistency of $\hat{\phi}$, it is necessary to show that $\parallel H(\hat{\phi},\eta_0)\parallel=O_p(n^{-1/2})$. We can magnify it by the triangle inequality as
\begin{eqnarray*}
\parallel H(\hat{\phi},\eta_0)\parallel&\leq&\parallel H(\hat{\phi},\eta_0)-H(\hat{\phi},\tilde{\eta})\parallel+\parallel H(\hat{\phi},\tilde{\eta})-H(\phi_0,\eta_0)-H_n(\hat{\phi},\tilde{\eta})+H_n(\phi_0,\eta_0)\parallel\\
&&+\parallel H_n(\hat{\phi},\tilde{\eta})\parallel+\parallel H(\phi_0,\eta_0)\parallel+\parallel H_n(\phi_0,\eta_0)\parallel\\
&\triangleq& A_1+A_2+A_3+A_4+A_5.
\end{eqnarray*}
It is obvious that $A_4+A_5=O_p(n^{-1})$, since $\parallel H(\phi_0,\eta_0)\parallel=\parallel H_n(\phi_0,\eta_0)\parallel=O_p(b^m)=O_p(n^{-1})$ under Assumption [B2]. Denote the partial derivative of $H(\phi,\eta)$ to $\eta$ by
\begin{eqnarray*}
\Gamma_2(\phi,\eta)=\frac{\partial H(\phi,\eta)}{\partial \eta^T}=\mathbf{E}\left[z_t(\eta)\left(\frac{d\sigma_t(\eta)}{d\eta^T}-\phi^T\frac{d z_t(\eta)}{d\eta^T}\right)\right]-\mathbf{E}\left[\left(\sigma_{t}(\eta)-\phi^Tz_t(\eta)\right)\frac{dz_t(\eta)}{d\eta^T}\right],
\end{eqnarray*}
which is finite at $(\phi_0,\eta_0)$ since
\begin{eqnarray*}
\Gamma_2(\phi_0,\eta_0)&\approx&\mathbf{E}\left[z_t(\eta)\left(\frac{d\sigma_t(\eta)}{d\eta^T}-\phi^T\frac{d z_t(\eta)}{d\eta^T}\right)\right]|_{\phi=\phi_0,\eta=\eta_0}\\
&=&\mathbf{E}\left[z_t(\eta)\left(\frac{d\sigma_t(\eta)}{d\eta^T}-\sum_{j=1}^{p}\beta_j\frac{d\sigma_{t-j}(\eta)}{d\eta^T}\right)\right]|_{\phi=\phi_0,\eta=\eta_0}
\end{eqnarray*}
and we denote $\Gamma_2(\phi_0,\eta_0)$ by $\Gamma_{20}$. Thus, $A_1$ can be bounded by
\begin{eqnarray*}
A_1&\leq&\parallel H(\hat{\phi},\tilde{\eta})-H(\hat{\phi},\eta_0)-\Gamma_2(\hat{\phi},\eta_0)(\tilde{\eta}-\eta_0)\parallel\\
&&+\parallel\Gamma_2(\hat{\phi},\eta_0)(\tilde{\eta}-\eta_0)-\Gamma_2(\phi_0,\eta_0)(\tilde{\eta}-\eta_0)\parallel+\parallel\Gamma_2(\phi_0,\eta_0)(\tilde{\eta}-\eta_0)\parallel\\
&=&O_p(\parallel\tilde{\eta}-\eta_0\parallel^2)+O_p(\parallel\hat{\phi}-\phi_0\parallel\parallel\tilde{\eta}-\eta_0\parallel)+\parallel\Gamma_2(\phi_0,\eta_0)(\tilde{\eta}-\eta_0)\parallel\\
&=&O_p(\parallel\tilde{\eta}-\eta_0\parallel)=O_p(n^{-1/2}).
\end{eqnarray*}
Additionally, it deduces that $\parallel H(\hat{\phi},\eta_0)\parallel\leq\parallel H(\hat{\phi},\tilde{\eta})\parallel(1+o_p(1))$.

For the term $A_2$, since $H(\phi,\eta)$ has stochastic equicontinuity under Assumptions [A1]-[A3], we can also obtain the following result using the Lemma 4.2 of \cite{chen_nonparametric_2008}
\begin{eqnarray*}
A_2&=&\parallel H(\hat{\phi},\tilde{\eta})-H(\phi_0,\eta_0)-H_n(\hat{\phi},\tilde{\eta})+H_n(\phi_0,\eta_0)\parallel\\
&\leq&\left(\parallel H(\hat{\phi},\tilde{\eta})\parallel+\parallel H_n(\hat{\phi},\tilde{\eta})\parallel\right)\times o_p(1)\\
&\leq&\left(\parallel H(\hat{\phi},\eta_0)\parallel(1+o_p(1))+\parallel H_n(\hat{\phi},\tilde{\eta})\parallel\right)\times o_p(1).
\end{eqnarray*}
Now, we have
\begin{eqnarray*}
\parallel H(\hat{\phi},\eta_0)\parallel\leq\left(\parallel H(\hat{\phi},\eta_0)\parallel(1+o_p(1))+\parallel H_n(\hat{\phi},\tilde{\eta})\parallel\right)\times o_p(1)+\parallel H_n(\hat{\phi},\tilde{\eta})\parallel+O_p(n^{-1/2}),
\end{eqnarray*}
which implies that
\begin{eqnarray*}
\parallel H(\hat{\phi},\eta_0)\parallel\times(1-o_p(1))\leq\parallel H_n(\hat{\phi},\tilde{\eta})\parallel\times(1+o_p(1))+O_p(n^{-1/2}).
\end{eqnarray*}
Note that $\parallel H_n(\phi_0,\tilde{\eta})\parallel=O_p(n^{-1/2})$, so we have
\begin{eqnarray*}
\parallel H_n(\hat{\phi},\tilde{\eta})\parallel=\min_{\phi}\parallel H_n(\phi,\tilde{\eta})\parallel\leq\parallel H_n(\phi_0,\tilde{\eta})\parallel=O_p(n^{-1/2}).
\end{eqnarray*}
Thus, $\parallel H(\hat{\phi},\eta_0)\parallel\leq O_p(n^{-1/2})$ and the $\sqrt{n}$ consistency of $\hat{\phi}$ is obtained.

To prove the asymptotic normality, we define
\begin{eqnarray*}
L_n(\phi,\tilde{\eta})=H_n(\phi_0,\eta_0)+H(\phi,\eta_0)+\Gamma_2(\phi_0,\eta_0)(\tilde{\eta}-\eta_0).
\end{eqnarray*}
Based on the $\sqrt{n}$ consistency of $\hat{\phi}$, we have
\begin{eqnarray*}
\parallel H_n(\hat{\phi},\tilde{\eta})-L_n(\hat{\phi},\tilde{\eta})\parallel&\leq&\parallel H_n(\hat{\phi},\tilde{\eta})-H(\hat{\phi},\tilde{\eta})-H_n(\phi_0,\eta_0)+H(\phi_0,\eta_0)\parallel\\
&&+\parallel H(\phi_0,\eta_0)\parallel+\parallel H(\hat{\phi},\eta_0)-H(\phi_0,\eta_0)-\Gamma_{10}(\hat{\phi}-\phi_0)\parallel\\
&&+\parallel H(\hat{\phi},\tilde{\eta})-H(\hat{\phi},\eta_0)-\Gamma_{2}(\hat{\phi},\eta_0)(\tilde{\eta}-\eta_0)\parallel\\
&&+\parallel \Gamma_{2}(\hat{\phi},\eta_0)(\tilde{\eta}-\eta_0)-\Gamma_{2}(\phi_0,\eta_0)(\tilde{\eta}-\eta_0)\parallel\\
&=&o_p(n^{-1/2}).
\end{eqnarray*}
Therefore,
\begin{eqnarray*}
L_n(\hat{\phi},\tilde{\eta})&=&\Gamma_{10}(\hat{\phi}-\phi_{0})+H_n(\phi_0,\eta_0)+H_n(\phi_0,\eta_0)+\Gamma_{2}(\phi_0,\eta_0)(\tilde{\eta}-\eta_0)\\
&=&H_n(\hat{\phi},\tilde{\eta})+o_p(n^{-1/2})=o_p(n^{-1/2}).
\end{eqnarray*}
Recall the fact that $\parallel H(\phi_0,\eta_0)\parallel=\parallel H_n(\phi_0,\eta_0)\parallel)=O_p(n^{-1})$, and then we have
\begin{eqnarray*}
\sqrt{n}(\hat{\phi}-\phi_0)=-\Gamma_{10}^{-1}\Gamma_{20}(\tilde{\eta}-\eta_0)+o_p(1).
\end{eqnarray*}
which deduces the asymptotic normality of $\hat{\phi}$ in Theorem \ref{thm3}.
\end{proof}

\subsection{Proof of Corollary \ref{cor3}}
It is trivial to deduce Corollary \ref{cor3} from the result of Theorem \ref{thm3}, so we omit it.

\subsection{Proof of Lemma \ref{empcon}}
\begin{proof}
This lemma actually is a particular case of Theorem 1 in \cite{wang_conditional_2016}. Denote
\begin{eqnarray*}
Y_t=\bar{G}(\phi_0,\varepsilon_t,z_t)\triangleq\left(\beta_{0}+\sum_{i=1}^{p}\beta_{i}\sigma_{t-i}+\sum_{j=1}^{q}\gamma_{j}|Y_{t-j}|\right)\varepsilon_t,
\end{eqnarray*}
which is strictly increasing in $\varepsilon_t$, and
\begin{eqnarray*}
\varepsilon_t=\bar{H}(\phi_0,Y_t,z_t)\triangleq\frac{Y_t}{\beta_{0}+\sum_{i=1}^{p}\beta_{i}\sigma_{t-i}+\sum_{j=1}^{q}\gamma_{j}|Y_{t-j}|}.
\end{eqnarray*}
In fact, the estimated innovation is obtained from $\hat{\varepsilon}_t=\bar{H}(\hat{\phi},Y_t,z_t)$. Since the $\sqrt{n}$ consistency of $\hat{\phi}$ has been proved, we only need to check the assumption 3 and the assumption 4 in \cite{wang_conditional_2016}.

Assumption 3(\rmnum{1}) is obvious under Assumption [C1], and assumption 3(\rmnum{2}) can be obtained from the result about mixing properties of GARCH precess in \cite{boussama_ergodicite_1998}; see also in \cite{Lindner2009}. Assumption 4(\rmnum{1}), 4(\rmnum{2}) and 4(\rmnum{4}) are hold since $\bar{G}$ is truly continuously differentiable on $\phi$ and $\varepsilon$, and Assumption [C2] ensures that $f_\varepsilon$ is bounded. Let $\dot{\bar{H}}(\phi,Y_t,z_t)$ be the partial derivative with respect to $\phi$. We have
\begin{eqnarray*}
\sup_{|v|<\delta}|\dot{\bar{H}}(\phi_0+v,Y_t,z_t)|=\frac{|Y_t| z_t}{[(\phi_0+v)^Tz_t]^2}\leq\frac{|Y_t| z_t}{c_{\beta}}
\end{eqnarray*}
for some constant $c_{\beta}>0$, which implies that assumption 4(\rmnum{3}) is hold.

Thus, based on the result of Theorem 1 in \cite{wang_conditional_2016}, we have
\begin{eqnarray}\label{a10}
\sup_{|x|<C}\left|\hat{F}_n(x)-\frac{1}{n}\sum_{i=1}^{n}\mathbf{1}_{\varepsilon_i<z}-f_\varepsilon(x)\mathbf{E}\left[\frac{z_t\varepsilon_t}{\phi_0^T z_t}\right](\hat{\phi}-\phi_0)\right|\leq o_p(n^{-1/2}),
\end{eqnarray}
for any given $C>0$. Moreover,
\begin{eqnarray*}
\sup_{|x|<C}|\hat{F}_n(x)-F_\varepsilon(x)|\leq O_p(n^{-1/2}),
\end{eqnarray*}
since the $\sqrt{n}$ convergence rate of the empirical distribution function $\hat{F}_n(x)$ and the $\sqrt{n}$ consistency of $\hat{\phi}$.
\end{proof}

\subsection{Proof of Theorem \ref{th4}}
The proof of Theorem \ref{th4} follows a same manner as the proof of Theorem 1 of \cite{peng_empirical_2015}, so we need some notations and lemmas first. Similarly as in section 3.2, we denote $(\mu_{\tau}(\varepsilon),\tau)$ by $(\mu_0,\tau_0)$.
\begin{lemapp}\label{lemma:Y1 and Y2}
Denote $\Delta_{\mu}=\mu-\mu_0$ and $\Delta_{\tau}=\tau-\tau_0$. Under the assumptions of Theorem \ref{th4}, when $\Delta\triangleq\parallel(\mu,\tau)-(\mu_0,\tau_0)\parallel$ converges to $0$ as $n\rightarrow\infty$, we have
\begin{eqnarray*}
\frac{1}{n}\sum_{i=1}^{n}W_{i1}(\mu,\tau)&=&
\frac{1}{n}\sum_{i=1}^{n}W_{i1}(\mu_0,\tau_0)-\left(F_\varepsilon(\mu_0)+\frac{\tau_0}{1-2\tau_0}\right)
\Delta_{\mu}-\frac{\mu_0}{(1-2\tau_0)^2}\Delta_{\tau}+o_p(\Delta),\\
\frac{1}{n}\sum_{i=1}^{n}W_{i2}(\mu,\tau)&=&
\frac{1}{n}\sum_{i=1}^{n}W_{i2}(\mu_0,\tau_0)+f_\varepsilon(\mu_0)\Delta_{\mu}+o_p(\Delta)+o_p(n^{-1/3}).
\end{eqnarray*}
\end{lemapp}
\begin{proof}We expand $\sum_{i=1}^{n}W_{i1}(m,\theta)$ as follows:
\begin{eqnarray*}
\frac{1}{n}\sum_{i=1}^{n}W_{i1}(\mu,\tau)&=&\int_{-\infty}^{\mu}(x-\mu)d\hat{F}_n(x)+\frac{\tau}{1-2\tau}\int_{-\infty}^{+\infty}xd\hat{F}_n(x)-\frac{\tau}{1-2\tau}\mu\\
&=&\int_{-\infty}^{\mu_0}(x-\mu_0)d\hat{F}_n(x)+\frac{\tau_0}{1-2\tau_0}\int_{-\infty}^{+\infty}xd\hat{F}_n(x)
-\frac{\tau_0}{1-2\tau_0}\mu_0+\int_{\mu_0}^{\mu}(x-\mu)d\hat{F}_n(x)\\
&&+\int_{-\infty}^{\mu_0}(\mu_0-\mu)d\hat{F}_n(x)+(\frac{\tau}{1-2\tau}-\frac{\tau_0}{1-2\tau_0})\int_{-\infty}^{+\infty}xd\hat{F}_n(x)+(\frac{\tau_0}{1-2\tau_0}\mu_0-\frac{\tau}{1-2\tau}\mu)\\
&=&\frac{1}{n}\sum_{i=1}^{n}W_{i1}(\mu_0,\tau_0)+\int_{\mu_0}^{\mu}(x-\mu)d\hat{F}_n(x)+\int_{-\infty}^{\mu_0}(\mu_0-\mu)d\hat{F}_n(x)\\
&&+\frac{\Delta_{\tau}}{(1-2\tau_0)^2}\int_{-\infty}^{+\infty}xd\hat{F}_n(x)+\frac{\mu_0\Delta_{\tau}}{(1-2\tau_0)^2}+\frac{\tau_0\Delta_{\mu}}{1-2\tau_0}+o_p(\Delta)\\
&=&\frac{1}{n}\sum_{i=1}^{n}W_{i1}(\mu_0,\tau_0)+\int_{\mu_0}^{\mu}(x-\mu)dF_\varepsilon(x)+\int_{\mu_0}^{\mu}(x-\mu)d(\hat{F}_n(x)-F_\varepsilon(x))\\
&&+\int_{-\infty}^{\mu_0}(\mu_0-\mu)dF_\varepsilon(x)+\int_{-\infty}^{\mu_0}(\mu_0-\mu)d(\hat{F}_n(x)-F_\varepsilon(x))\\
&&+\frac{\Delta_{\tau}}{(1-2\tau_0)^2}\int_{-\infty}^{+\infty}xdF_\varepsilon(x)+\frac{\Delta_{\tau}}{(1-2\tau_0)^2}\int_{-\infty}^{+\infty}xd(\hat{F}_n(x)-F_\varepsilon(x))\\
&&+\frac{\mu_0\Delta_{\tau}}{(1-2\tau_0)^2}+\frac{\tau_0\Delta_{\mu}}{1-2\tau_0}+o_p(\Delta)\\
&=&\frac{1}{n}\sum_{i=1}^{n}W_{i1}(\mu_0,\tau_0)+\int_{\mu_0}^{\mu}(x-\mu)dF_\varepsilon(x)+\int_{\mu_0}^{\mu}(x-\mu)d(\hat{F}_n(x)-F_\varepsilon(x))\\
&&-F(\mu_0)\Delta_{\mu}-(\hat{F}_n(\mu_0)-F_\varepsilon(\mu_0))\Delta_{\mu}-\frac{\mu_0\Delta_{\tau}}{(1-2\tau_0)^2}-\frac{\tau_0\Delta_{\mu}}{1-2\tau_0}+o_p(\Delta)
\end{eqnarray*}
According to the mean value theorem of integrals, there exist $\mu_1$, $\mu_2$ between $\mu$ and $\mu_0$ such that
$$\int_{\mu_0}^{\mu}(x-\mu)dF_\varepsilon(x)=(\mu_1-\mu_0)(F_\varepsilon(\mu)-F_\varepsilon(\mu_0)),$$
$$\int_{\mu_0}^{\mu}(x-\mu)d(\hat{F}_n(x)-F_\varepsilon(x))=(\mu_2-\mu_0)[(\hat{F}_n(\mu)-\hat{F}_n(\mu_0))-(F_\varepsilon(\mu)-F_\varepsilon(\mu_0))].$$

Based on the fact that $F$ has bounded derivative and the conclusion of Lemma \ref{empcon}, we have
$$\int_{\mu_0}^{\mu}(x-\mu)dF_\varepsilon(x)=o_p(\Delta),$$
$$\int_{\mu_0}^{\mu}(x-\mu)d(\hat{F}_n(x)-F_\varepsilon(x))=o_p(\Delta).$$
Therefore, we can rewrite $\frac{1}{n}\sum_{i=1}^{n}W_{i1}(\mu,\tau)$ as:
\begin{eqnarray*}
\frac{1}{n}\sum_{i=1}^{n}W_{i1}(\mu,\tau)=\frac{1}{n}\sum_{i=1}^{n}W_{i1}(\mu_0,\tau_0)-\left(F_\varepsilon(\mu_0)+\frac{\tau_0}{1-2\tau_0}\right)
\Delta_{\mu}-\frac{\mu_0}{(1-2\tau_0)^2}\Delta_{\tau}+o_p(\Delta)
\end{eqnarray*}
We can obtain the following expansion of $\sum_{i=1}^{n}W_{i2}(\mu,\tau)$ in the same way,
\begin{eqnarray*}
\frac{1}{n}\sum_{i=1}^{n}W_{i2}(\mu,\tau)&=&\frac{1}{n}\sum_{i=1}^{n}W_{i2}(\mu_0,\tau_0)+(F_\varepsilon(m)-F_\varepsilon(m_\alpha))+(\hat{F}_n(\mu)-F_\varepsilon(\mu))-(\hat{F}_n(\mu_0)-F_\varepsilon(\mu_0))\\
&=&\frac{1}{n}\sum_{i=1}^{n}W_{i2}(\mu_0,\tau_0)+f_\varepsilon(\mu_0)\Delta_{\mu}+o_p(\Delta_{m}^2)+o_p(n^{-1/3}).
\end{eqnarray*}
\end{proof}

\begin{lemapp}\label{lemma:normality convergence}
Under the assumptions in Theorem \ref{th4}, as $n$ approach infinity, we have
\begin{eqnarray}
\frac{1}{\sqrt{n}}\sum_{i=1}^{n}W_{i}(\mu_0,\tau_0)\stackrel{d}{\rightarrow}N(\mathbf{0},\Sigma_0),
\end{eqnarray}
\begin{eqnarray}
\frac{1}{\sqrt{n}}\sum_{i=1}^{n}W_{i}(\mu_0,\tau_0)=O_p((\log\log n)^{1/2}),
\end{eqnarray}
\begin{eqnarray}
\frac{1}{n}\sum_{i=1}^{n}W_{i}(\mu_0,\tau_0)W_{i}^T(\mu_0,\tau_0)\stackrel{p}{\rightarrow}\Sigma_0.
\end{eqnarray}
and the matrix $\Sigma_0$ is positive definite.
\end{lemapp}
\begin{proof}
These three formulas can be deviated from the central limit theorem, the law of iterated logarithmic and the weak law of large numbers because $\mathbf{E}[W_{i}(\mu_0,\tau_0)]=\mathbf{0}$ and $\mathbf{Var}[W_{i}(\mu_0,\tau_0)]=\Sigma_0$. To demonstrate $\Sigma_0$ is positive definite, we need to show that $VaR\left((a\ b)W_{i}(\mu_0,\tau_0)\right)>0$ for any $a^2+b^2>0$. If there exist constants $a$, $b$ such that $a^2+b^2>0$ and $VaR\left((a\ b)W_{i}(\mu_0,\tau_0)\right)=0$, then we have $$a\left[(\hat{\varepsilon}_i-\mu_0)I(X<\mu_0)+\frac{\tau_0}{1-2\tau_0}\hat{\varepsilon}_i-\frac{\tau_0}{1-2\tau_0}\mu_0\right]=b\left[I(\hat{\varepsilon}_i<\mu_0)\right]\quad a.s.$$
This implies $b\equiv a(X-\mu_0)$, which leads to a contradiction to that $a^2+b^2>0$. Hence, we can learn that $\Sigma_0$ is positive definite.
\end{proof}

\begin{lemapp} \label{lemma:interior point}
Under the assumptions of Theorem \ref{th4}, $(\hat{\mu}_{\tau_0},\hat{\tau}_0)$, the solution of the formula (\ref{hmutau}), is an interior point of the ball $\{(\mu,\tau):\Delta\le n^{-1/3}\}$.
\end{lemapp}
\begin{proof}
As the same as the proof in \cite{owen_empirical_1990}, the $\lambda$ in Eq. (\ref{inclem}) can be represented as,
\begin{eqnarray}
\lambda=\left[\frac{1}{n}\sum_{i=1}^{n}W_{i}(\mu,\tau)W_{i}^T(\mu,\tau)\right]^{-1}
\left(\frac{1}{n}\sum_{i=1}^{n}W_{i}(\mu,\tau)\right)+o_p(n^{-1/3}),
\end{eqnarray}
and
\begin{eqnarray}
\frac{1}{n}\sum_{i=1}^{n}W_{i}(\mu,\tau)W_{i}^T(\mu,\tau)\stackrel{p}{\rightarrow}\Sigma_0,
\end{eqnarray}
uniformly in $\{(\mu,\tau):\Delta\le n^{-1/3}\}$. Based on the above fact and Lemma \ref{lemma:normality convergence}, the Taylor expansion of $l(\mu,\tau)$ can be expressed as
\begin{eqnarray*}
l(\mu,\tau)&=&2\sum_{i=1}^{n}\lambda^{T}W_{i}(\mu,\tau)-\sum_{i=1}^{n}\lambda^{T}W_{i}(\mu,\tau)W_{i}^T(\mu,\tau)\lambda+o_p(n^{-1/3})\\
&=&n[\frac{1}{n}\sum_{i=1}^{n}W_{i}(\mu,\tau)]^{T}[\frac{1}{n}\sum_{i=1}^{n}W_{i}(\mu,\tau)W_{i}^T(\mu,\tau)]^{-1}[\frac{1}{n}\sum_{i=1}^{n}W_{i}(\mu,\tau)]+o_p(n^{-1/3})\\
&=&n[\frac{1}{n}\sum_{i=1}^{n}W_{i}(\mu_0,\tau_0)+O_p(\Delta)]^{T}\Sigma_0^{-1}[\frac{1}{n}\sum_{i=1}^{n}W_{i}(\mu_0,\tau_0)+O_p(\Delta)]+o_p(n^{-1/3})\\
&\geq&n[O_p(n^{-1/2}(\log\log n)^{1/2})+O_p(n^{-1/3})]^{T}\Sigma_0^{-1}[O_p(n^{-1/2}(\log\log n)^{1/2})+O_p(n^{-1/3})]\\
&&+o_p(n^{-1/3})\\
&=&(c-\delta)n^{1/3}
\end{eqnarray*}
where $c$ is a positive constant since $\Sigma_0$ is of positive defined, and $\delta$ is infinitesimal. Meanwhile,
\begin{eqnarray*}
l(\mu_0,\tau_0)&=&\left[\frac{1}{\sqrt{n}}\sum_{i=1}^{n}W_{i}(\mu_0,\tau_0)\right]^{T}\left[\frac{1}{n}
\sum_{i=1}^{n}W_{i}(\mu_0,\tau_0)W_{i}^T(\mu_0,\tau_0)\right]^{-1}
\left[\frac{1}{\sqrt{n}}\sum_{i=1}^{n}W_{i}(\mu_0,\tau_0)\right]+o_p(1)\\
&=&O_p(\log\log n)<n^{1/3}.
\end{eqnarray*}
Since $l(\mu,\tau)$ is continuous in the ball $\{(\mu,\tau):\Delta\le n^{-1/3}\}$, the solution of (\ref{hmutau}) must be in the interior of this ball.
\end{proof}

\begin{proof}[proof of Theorem \ref{th4}]Lemma \ref{lemma:normality convergence} shows that $\parallel(\mu_0,\tau_0)-(\mu,\tau)\parallel\leq n^{-1/3}$. Additionally, in the ball $\{(\mu,\tau):\Delta\le n^{-1/3}\}$, we have
\begin{eqnarray*}
l(\mu,\tau)&=&n\left[\frac{1}{n}\sum_{i=1}^{n}W_{i}(\mu,\tau)\right]^{T}\left[\frac{1}{n}\sum_{i=1}^{n}W_{i}(\mu,\tau)W_{i}^T(\mu,\tau)\right]^{-1}
\left[\frac{1}{n}\sum_{i=1}^{n}W_{i}(\mu,\tau)\right]+o_p(n^{-1/3})\\
&=&\left[\frac{1}{\sqrt{n}}\sum_{i=1}^{n}W_{i}(\mu_0,\tau_0)+\Sigma_{1}v\right]^{T}\Sigma_0^{-1}\left[\frac{1}{\sqrt{n}}
\sum_{i=1}^{n}W_{i}(\mu_0,\tau_0)+\Sigma_{1}v\right]+o_p(n^{-1/3})
\end{eqnarray*}
where $v=\sqrt{n}(\mu-\mu_0,\tau-\tau_0)^T$ and $\Sigma_1=\left[\begin{array}{cc}-(F_\varepsilon(\mu_{\tau_0})+\frac{\tau_0}{1-2\tau_0})&\frac{-\mu_{\tau_0}}{(1-2\tau_0)^2}\\f_\varepsilon(\mu_{\tau_0})
&0\end{array}\right]$. Hence, $l(\mu,\tau)$ will be minimized at
\begin{eqnarray}\label{a15}
\Sigma_{1}v=-\frac{1}{\sqrt{n}}\sum_{i=1}^{n}W_{i}(\mu_0,\tau_0)\stackrel{d}{\rightarrow}N(\mathbf{0},\Sigma_0)
\end{eqnarray}
except for $\mu_0=0$ or $f(\mu_0)=0$. Hence, we have
\begin{eqnarray*}
\sqrt{n}\left(\begin{array}{c}\hat{\mu}-\mu_0 \\ \hat{\tau}-\tau_0 \end{array}\right)\stackrel{d}{\rightarrow}N\left(\mathbf{0},\Sigma_1^{-1}\Sigma_0(\Sigma_1^{-1})^T\right).
\end{eqnarray*}
\end{proof}

\subsection{Proof of Theorem \ref{th5}}
\begin{proof}
Since $\hat{\mu}$, $\hat{\tau}$, $\hat{\sigma}_T$ are all $\sqrt{n}$-consistency and asymptotic normally, the consistency of $\widehat{Q}_\alpha(Y_T|\mathcal{F}_{T-1})$ and $\widehat{ES}_\alpha(Y_T|\mathcal{F}_{T-1})$ are trivial. To proof their asymptotic normality, we first give the approximation expansion of $\hat{\mu}$ and $\hat{\tau}$ with respect to $\hat{\phi}-\phi_0$. Based on the Eq. (\ref{a15}), we have
\begin{eqnarray*}
\sqrt{n}\left(\begin{array}{c}\hat{\mu}-\mu_0 \\ \hat{\tau}-\tau_0 \end{array}\right)&=&\frac{1}{\sqrt{n}}\Sigma_1^{-1}\sum_{i=1}^{n}W_{i}(\mu_0,\tau_0)\\
&=&\sqrt{n}\Sigma_1^{-1}\left[\begin{array}{c}\frac{1}{n}\sum_{i=1}^{n}\left((\hat{\varepsilon}_i-\mu_0)I(\hat{\varepsilon}_i<\mu_0)+\frac{\tau_0}{1-2\tau_0}(\hat{\varepsilon}_i-\mu_0)\right) \\ \frac{1}{n}\sum_{i=1}^{n}I(\hat{\varepsilon}_i<\mu_0)-\alpha \end{array}\right]\\
&=&\sqrt{n}\Sigma_1^{-1}\left[\begin{array}{c}\int_{-\infty}^{\mu_0}(x-\mu_0)d\hat{F}_n(x)+\frac{\tau_0}{1-2\tau_0}\int_{-\infty}^{+\infty}(x-\mu_0)d\hat{F}_n(x)
\\\int_{-\infty}^{\mu_0}d\hat{F}_n(x)-\alpha\end{array}\right]\\
&=&\sqrt{n}\Sigma_1^{-1}\left[\begin{array}{c}\int_{-\infty}^{\mu_0}(x-\mu_0)d(\hat{F}_n(x)-F_{\varepsilon}(x))+\frac{\tau_0}{1-2\tau_0}\int_{-\infty}^{+\infty}(x-\mu_0)d(\hat{F}_n(x)-F_{\varepsilon}(x))
\\\int_{-\infty}^{\mu_0}d(\hat{F}_n(x)-F_{\varepsilon}(x))\end{array}\right]\\
&=&\sqrt{n}\Sigma_1^{-1}\left[\begin{array}{c}\int_{-\infty}^{\mu_0}(\hat{F}_n(x)-F_{\varepsilon}(x))dx+\frac{\tau_0}{1-2\tau_0}\int_{-\infty}^{+\infty}(\hat{F}_n(x)-F_{\varepsilon}(x))dx
\\\hat{F}_n(\mu_0)-F_{\varepsilon}(\mu_0)\end{array}\right].
\end{eqnarray*}
According to Eq.(\ref{a10}) in the proof of Lemma \ref{empcon}, we have
\begin{eqnarray*}
\sqrt{n}\left(\begin{array}{c}\hat{\mu}-\mu_0 \\ \hat{\tau}-\tau_0 \end{array}\right)&=&
\sqrt{n}\Sigma_1^{-1}\left[\begin{array}{c}\left(\int_{-\infty}^{\mu_0}f_{\varepsilon}(x)dx+\frac{\tau_0}{1-2\tau_0}\int_{-\infty}^{+\infty}f_{\varepsilon}(x)dx\right)\cdot\mathbf{E}\left[\frac{z_t\varepsilon_t}{\phi_0^T z_t}\right](\hat{\phi}-\phi_0)+o_p(\frac{1}{\sqrt{n}})\\
f_{\varepsilon}(\mu_0)\cdot\mathbf{E}\left[\frac{z_t\varepsilon_t}{\phi_0^T z_t}\right](\hat{\phi}-\phi_0)+o_p(\frac{1}{\sqrt{n}})\end{array}\right]\\
&=&\Sigma_1^{-1}\left[\begin{array}{c}e_1\\e_2\end{array}\right]\Sigma_2\cdot\sqrt{n}(\hat{\phi}-\phi_0)+o_p(1),
\end{eqnarray*}
where $e_1=\int_{-\infty}^{\mu_0}f_{\varepsilon}(x)dx+\frac{\tau_0}{1-2\tau_0}\int_{-\infty}^{+\infty}f_{\varepsilon}(x)dx$, $e_2=f_{\varepsilon}(\mu_0)$ and $\Sigma_2=\mathbf{E}\left[\frac{z_t\varepsilon_t}{\phi_0^Tz_t}\right]$.
Since $\widehat{Q}_\alpha(\varepsilon)=\hat{\mu}$, it can be expanded as
\begin{eqnarray}
\sqrt{n}(\widehat{Q}_\alpha(\varepsilon)-Q_\alpha(\varepsilon))=
{\left[\begin{array}{c} 1 \\ 0\end{array}\right]}^T\Sigma_1^{-1}\left[\begin{array}{c}e_1\\e_2\end{array}\right]\Sigma_2\cdot\sqrt{n}(\hat{\phi}-\phi_0)+o_p(1).
\end{eqnarray}
For $\widehat{ES}_\alpha(\varepsilon)=\left(1+\frac{\hat{\tau}}{(1-2\hat{\tau})\alpha}\right)\hat{\mu}-\frac{\hat{\tau}}{(n-m)(1-2\hat{\tau})\alpha}\sum_{t}\hat{\varepsilon}_t$, through Taylor expansion, we have
\begin{eqnarray}
\sqrt{n}(\widehat{ES}_\alpha(\varepsilon)-ES_\alpha(\varepsilon))=
{\left[\begin{array}{c} e_3 \\ e_4\end{array}\right]}^T\Sigma_1^{-1}\left[\begin{array}{c}e_1\\e_2\end{array}\right]\Sigma_2\cdot\sqrt{n}(\hat{\phi}-\phi_0)+o_p(1),
\end{eqnarray}
where $e_3=1+\frac{\tau_0}{(1-2\tau_0)\alpha}$ and $e_4=\frac{\mu_0-\mathbf{E}[\varepsilon_t]}{\alpha(1-2\tau_0)^2}$.

Hence, $\widehat{Q}_\alpha(\varepsilon)$ and $\widehat{ES}_\alpha(\varepsilon)$ are both asymptotic normality. In addition, conditional on information prior to time $T$, we have the following expansion of $\hat{\sigma}_T$,
\begin{eqnarray}
\sqrt{n}(\hat{\sigma}_T-\sigma_T)=
z_t^{T}\sqrt{n}(\hat{\phi}-\phi_0)+o_p(1),
\end{eqnarray}
By delta method and Slutsky's theorem, conditional on information prior to time $T$, $\widehat{Q}_\alpha(Y_T|\mathcal{F}_{T-1})$ and $\widehat{ES}_\alpha(Y_T|\mathcal{F}_{T-1})$ are asymptotic normality, and their asymptotic variance can be expressed as
\begin{eqnarray*}
AVar(\widehat{Q}_\alpha(Y_T|\mathcal{F}_{T-1}))=(\sigma_T)^2AVar(\widehat{Q}_\alpha(\varepsilon))+(Q_\alpha(\varepsilon))^2AVar(\hat{\sigma}_T)+2\sigma_TQ_\alpha(\varepsilon)ACov(\hat{\sigma}_T,\widehat{Q}_\alpha(\varepsilon)),
\end{eqnarray*}
\begin{eqnarray*}
AVar(\widehat{ES}_\alpha(Y_T|\mathcal{F}_{T-1}))=(\sigma_T)^2AVar(\widehat{ES}_\alpha(\varepsilon))+(ES_\alpha(\varepsilon))^2AVar(\hat{\sigma}_T)+2\sigma_T ES_\alpha(\varepsilon)ACov(\hat{\sigma}_T,\widehat{ES}_\alpha(\varepsilon)).
\end{eqnarray*}
Denote $\Lambda_1={\left[\begin{array}{c} 1 \\ 0\end{array}\right]}^T\Sigma_1^{-1}\left[\begin{array}{c}e_1\\e_2\end{array}\right]\Sigma_2$ and $\Lambda_2={\left[\begin{array}{c} e_3 \\ e_4\end{array}\right]}^T\Sigma_1^{-1}\left[\begin{array}{c}e_1\\e_2\end{array}\right]\Sigma_2$, and then we have
\begin{eqnarray}
AVar(\widehat{Q}_\alpha(\varepsilon))&=&\Lambda_1\Xi_{\phi}\Lambda_1^T,\\
ACov(\hat{\sigma}_T,\widehat{Q}_\alpha(\varepsilon))&=&\Lambda_1\Xi_{\phi}z_T,\\
AVar(\widehat{ES}_\alpha(\varepsilon))&=&\Lambda_2\Xi_{\phi}\Lambda_2^T,\\
ACov(\hat{\sigma}_T,\widehat{ES}_\alpha(\varepsilon))&=&\Lambda_2\Xi_{\phi}z_T.
\end{eqnarray}
Hence, the asymptotic variance of $\widehat{Q}_\alpha(Y_T|\mathcal{F}_{T-1})$ and $\widehat{ES}_\alpha(Y_T|\mathcal{F}_{T-1})$ are
\begin{eqnarray}
AVar(\widehat{Q}_\alpha(Y_T|\mathcal{F}_{T-1}))=(\sigma_T)^2\Lambda_1\Xi_{\phi}\Lambda_1^T+(Q_\alpha(\varepsilon))^2z_T^{T}\Xi_{\phi}z_T+2\sigma_T Q_\alpha(\varepsilon)\Lambda_1\Xi_{\phi}z_T,\\
AVar(\widehat{ES}_\alpha(Y_T|\mathcal{F}_{T-1}))=(\sigma_T)^2\Lambda_2\Xi_{\phi}\Lambda_2^T+(ES_\alpha(\varepsilon))^2z_T^{T}\Xi_{\phi}z_T+2\sigma_T ES_\alpha(\varepsilon)\Lambda_2\Xi_{\phi}z_T.
\end{eqnarray}
\end{proof}

\bibliographystyle{model2-names}
\bibliography{ESE}

\begin{thebibliography}{54}
\expandafter\ifx\csname natexlab\endcsname\relax\def\natexlab#1{#1}\fi
\expandafter\ifx\csname url\endcsname\relax
  \def\url#1{\texttt{#1}}\fi
\expandafter\ifx\csname urlprefix\endcsname\relax\def\urlprefix{URL }\fi
\providecommand{\eprint}[2][]{\url{#2}}
\providecommand{\bibinfo}[2]{#2}
\ifx\xfnm\relax \def\xfnm[#1]{\unskip,\space#1}\fi
\bibitem[{Acerbi et~al.(2001)Acerbi, Nordio and Sirtori}]{acerbi_expected_2001}
\bibinfo{author}{Acerbi, C.}, \bibinfo{author}{Nordio, C.},
  \bibinfo{author}{Sirtori, C.}, \bibinfo{year}{2001}.
\newblock \bibinfo{title}{Expected shortfall as a tool for financial risk
  management}.
\newblock \bibinfo{journal}{arXiv:cond-mat/0102304} \bibinfo{note}{ArXiv:
  cond-mat/0102304}.
\bibitem[{Acerbi and Szekely(2014)}]{acerbi2014back}
\bibinfo{author}{Acerbi, C.}, \bibinfo{author}{Szekely, B.},
  \bibinfo{year}{2014}.
\newblock \bibinfo{title}{Back-testing expected shortfall}.
\newblock \bibinfo{journal}{Risk} , \bibinfo{pages}{76}.
\bibitem[{Acerbi and Tasche(2002)}]{acerbi_coherence_2002}
\bibinfo{author}{Acerbi, C.}, \bibinfo{author}{Tasche, D.},
  \bibinfo{year}{2002}.
\newblock \bibinfo{title}{On the coherence of expected shortfall}.
\newblock \bibinfo{journal}{Journal of Banking \& Finance}
  \bibinfo{volume}{26}, \bibinfo{pages}{1487--1503}.
\bibitem[{Aigner et~al.(1976)Aigner, Amemiya and
  Poirier}]{aigner_estimation_1976}
\bibinfo{author}{Aigner, D.J.}, \bibinfo{author}{Amemiya, T.},
  \bibinfo{author}{Poirier, D.J.}, \bibinfo{year}{1976}.
\newblock \bibinfo{title}{On the {Estimation} of {Production} {Frontiers}:
  {Maximum} {Likelihood} {Estimation} of the {Parameters} of a {Discontinuous}
  {Density} {Function}}.
\newblock \bibinfo{journal}{International Economic Review}
  \bibinfo{volume}{17}, \bibinfo{pages}{377--96}.
\bibitem[{Artzner et~al.(1999)Artzner, Delbaen, Eber and
  Heath}]{artzner_coherent_1999}
\bibinfo{author}{Artzner, P.}, \bibinfo{author}{Delbaen, F.},
  \bibinfo{author}{Eber, J.M.}, \bibinfo{author}{Heath, D.},
  \bibinfo{year}{1999}.
\newblock \bibinfo{title}{Coherent {Measures} of {Risk}}.
\newblock \bibinfo{journal}{Mathematical Finance} \bibinfo{volume}{9},
  \bibinfo{pages}{203--228}.
\bibitem[{Bellini and Bernardino(2017)}]{bellini_risk_2017}
\bibinfo{author}{Bellini, F.}, \bibinfo{author}{Bernardino, E.D.},
  \bibinfo{year}{2017}.
\newblock \bibinfo{title}{Risk management with expectiles}.
\newblock \bibinfo{journal}{The European Journal of Finance}
  \bibinfo{volume}{23}, \bibinfo{pages}{487--506}.
\bibitem[{Berkes et~al.(2003)Berkes, Horvath and Kokoszka}]{berkes_garch_2003}
\bibinfo{author}{Berkes, I.}, \bibinfo{author}{Horvath, L.},
  \bibinfo{author}{Kokoszka, P.}, \bibinfo{year}{2003}.
\newblock \bibinfo{title}{{GARCH} processes: structure and estimation}.
\newblock \bibinfo{journal}{Bernoulli} \bibinfo{volume}{9},
  \bibinfo{pages}{201--227}.
\bibitem[{Bollerslev(1986)}]{Bollerslev1986Generalized}
\bibinfo{author}{Bollerslev, T.P.}, \bibinfo{year}{1986}.
\newblock \bibinfo{title}{Generalized autoregressive conditional
  heteroskedasticity with applications in finance}.
\newblock \bibinfo{journal}{General Information} \bibinfo{volume}{31},
  \bibinfo{pages}{307¨C327}.
\bibitem[{Boussama(1998)}]{boussama_ergodicite_1998}
\bibinfo{author}{Boussama, F.}, \bibinfo{year}{1998}.
\newblock \bibinfo{title}{Ergodicity, mixing and estimation in GARCH models}.
\newblock \bibinfo{publisher}{Paris 7}.
\bibitem[{Broda and Paolella(2011)}]{broda_expected_2011}
\bibinfo{author}{Broda, S.A.}, \bibinfo{author}{Paolella, M.S.},
  \bibinfo{year}{2011}.
\newblock \bibinfo{title}{Expected shortfall for distributions in finance}, in:
  \bibinfo{booktitle}{Statistical {Tools} for {Finance} and {Insurance}}.
  \bibinfo{publisher}{Springer Berlin Heidelberg}, pp. \bibinfo{pages}{57--99}.
\bibitem[{Cai(2002)}]{cai_regression_2002}
\bibinfo{author}{Cai, Z.}, \bibinfo{year}{2002}.
\newblock \bibinfo{title}{{Regression} {Quantiles} {For} {Time} {Series}}.
\newblock \bibinfo{journal}{Econometric Theory} \bibinfo{volume}{18},
  \bibinfo{pages}{169--192}.
\bibitem[{Cai and Wang(2008)}]{cai_nonparametric_2008}
\bibinfo{author}{Cai, Z.}, \bibinfo{author}{Wang, X.}, \bibinfo{year}{2008}.
\newblock \bibinfo{title}{Nonparametric estimation of conditional {VaR} and
  expected shortfall}.
\newblock \bibinfo{journal}{Journal of Econometrics} \bibinfo{volume}{147},
  \bibinfo{pages}{120--130}.
\bibitem[{Cai and Xu(2009)}]{cai_nonparametric_2009}
\bibinfo{author}{Cai, Z.}, \bibinfo{author}{Xu, X.}, \bibinfo{year}{2009}.
\newblock \bibinfo{title}{Nonparametric {Quantile} {Estimations} for {Dynamic}
  {Smooth} {Coefficient} {Models}}.
\newblock \bibinfo{journal}{Journal of the American Statistical Association}
  \bibinfo{volume}{104}, \bibinfo{pages}{371--383}.
\bibitem[{Chen(2008)}]{chen_nonparametric_2008}
\bibinfo{author}{Chen, S.X.}, \bibinfo{year}{2008}.
\newblock \bibinfo{title}{Nonparametric {Estimation} of {Expected}
  {Shortfall}}.
\newblock \bibinfo{journal}{Journal of Financial Econometrics}
  \bibinfo{volume}{6}, \bibinfo{pages}{87--107}.
\bibitem[{Efron(1991)}]{Efron_1991}
\bibinfo{author}{Efron, B.}, \bibinfo{year}{1991}.
\newblock \bibinfo{title}{Regression percentiles using asymmetric squared error
  loss}.
\newblock \bibinfo{journal}{Statistica Sinica} \bibinfo{volume}{1},
  \bibinfo{pages}{93--125}.
\bibitem[{Engle and Manganelli(2004)}]{engle_caviar_2004}
\bibinfo{author}{Engle, R.F.}, \bibinfo{author}{Manganelli, S.},
  \bibinfo{year}{2004}.
\newblock \bibinfo{title}{{CAViaR}}.
\newblock \bibinfo{journal}{Journal of Business \& Economic Statistics}
  \bibinfo{volume}{22}, \bibinfo{pages}{367--381}.
\bibitem[{Fan and Yao(2006)}]{Fan2006Nonlinear}
\bibinfo{author}{Fan, J.}, \bibinfo{author}{Yao, Q.}, \bibinfo{year}{2006}.
\newblock \bibinfo{title}{Nonlinear Time Series: Nonparametric and Parametric
  Methods}.
\newblock \bibinfo{publisher}{Science Press}.
\bibitem[{Fissler and Ziegel(2016)}]{fissler_higher_2016}
\bibinfo{author}{Fissler, T.}, \bibinfo{author}{Ziegel, J.F.},
  \bibinfo{year}{2016}.
\newblock \bibinfo{title}{Higher order elicitability and osband's principle}.
\newblock \bibinfo{journal}{The Annals of Statistics} \bibinfo{volume}{44},
  \bibinfo{pages}{1680--1707}.
\bibitem[{Fissler et~al.(2015)Fissler, Ziegel and
  Gneiting}]{fissler_expected_2015}
\bibinfo{author}{Fissler, T.}, \bibinfo{author}{Ziegel, J.F.},
  \bibinfo{author}{Gneiting, T.}, \bibinfo{year}{2015}.
\newblock \bibinfo{title}{Expected {Shortfall} is jointly elicitable with
  {Value} at {Risk} - {Implications} for backtesting}.
\newblock \bibinfo{journal}{arXiv:1507.00244 [q-fin]} \bibinfo{note}{ArXiv:
  1507.00244}.
\bibitem[{Glosten et~al.(1993)Glosten, Jagannathan and Runkle}]{Glosten1993}
\bibinfo{author}{Glosten, L.R.}, \bibinfo{author}{Jagannathan, R.},
  \bibinfo{author}{Runkle, D.E.}, \bibinfo{year}{1993}.
\newblock \bibinfo{title}{On the relation between the expected value and the
  volatility of the nominal excess return on stocks}.
\newblock \bibinfo{journal}{Journal of Finance} \bibinfo{volume}{48},
  \bibinfo{pages}{1779--1801}.
\bibitem[{Granger and Poon(2003)}]{Granger2003Forecasting}
\bibinfo{author}{Granger, C.W.J.}, \bibinfo{author}{Poon, S.},
  \bibinfo{year}{2003}.
\newblock \bibinfo{title}{Forecasting financial market volatility: A review}.
\newblock \bibinfo{journal}{Ssrn Electronic Journal} .
\bibitem[{Hall and Yao(2003)}]{hall_inference_2003}
\bibinfo{author}{Hall, P.}, \bibinfo{author}{Yao, Q.}, \bibinfo{year}{2003}.
\newblock \bibinfo{title}{Inference in {Arch} and {Garch} {Models} with
  heavy-tailed {Errors}}.
\newblock \bibinfo{journal}{Econometrica} \bibinfo{volume}{71},
  \bibinfo{pages}{285--317}.
\bibitem[{Hang~Chan et~al.(2007)Hang~Chan, Deng, Peng and
  Xia}]{hang_chan_interval_2007}
\bibinfo{author}{Hang~Chan, N.}, \bibinfo{author}{Deng, S.J.},
  \bibinfo{author}{Peng, L.}, \bibinfo{author}{Xia, Z.}, \bibinfo{year}{2007}.
\newblock \bibinfo{title}{Interval estimation of value-at-risk based on {GARCH}
  models with heavy-tailed innovations}.
\newblock \bibinfo{journal}{Journal of Econometrics} \bibinfo{volume}{137},
  \bibinfo{pages}{556--576}.
\bibitem[{Harmantzis et~al.(2006)Harmantzis, Miao and
  Chien}]{fotios_c._harmantzis_empirical_2006}
\bibinfo{author}{Harmantzis, F.C.}, \bibinfo{author}{Miao, L.},
  \bibinfo{author}{Chien, Y.}, \bibinfo{year}{2006}.
\newblock \bibinfo{title}{Empirical study of value-at-risk and expected
  shortfall models with heavy tails}.
\newblock \bibinfo{journal}{The Journal of Risk Finance} \bibinfo{volume}{7},
  \bibinfo{pages}{117--135}.
\bibitem[{Huber(1967)}]{huber1967}
\bibinfo{author}{Huber, P.J.}, \bibinfo{year}{1967}.
\newblock \bibinfo{title}{The behavior of maximum likelihood estimates under
  nonstandard conditions}, in: \bibinfo{booktitle}{Proceedings of the Fifth
  Berkeley Symposium on Mathematical Statistics and Probability, Volume 1:
  Statistics}, \bibinfo{publisher}{University of California Press},
  \bibinfo{address}{Berkeley, Calif.}. pp. \bibinfo{pages}{221--233}.
\bibitem[{Jones(1994)}]{jones_expectiles_1994}
\bibinfo{author}{Jones, M.C.}, \bibinfo{year}{1994}.
\newblock \bibinfo{title}{Expectiles and {M}-quantiles are quantiles}.
\newblock \bibinfo{journal}{Statistics \& Probability Letters}
  \bibinfo{volume}{20}, \bibinfo{pages}{149--153}.
\bibitem[{Jorion(2000)}]{Jorion2000Value}
\bibinfo{author}{Jorion, P.}, \bibinfo{year}{2000}.
\newblock \bibinfo{title}{Value at risk: The new benchmark for controlling
  market risk}, in: \bibinfo{booktitle}{Irwin Professional Publishing, IL}.
\bibitem[{Kai et~al.(2010)Kai, Li and Zou}]{kai_local_2010}
\bibinfo{author}{Kai, B.}, \bibinfo{author}{Li, R.}, \bibinfo{author}{Zou, H.},
  \bibinfo{year}{2010}.
\newblock \bibinfo{title}{Local {Composite} {Quantile} {Regression}
  {Smoothing}: {An} {Efficient} and {Safe} {Alternative} to {Local}
  {Polynomial} {Regression}}.
\newblock \bibinfo{journal}{Journal of the Royal Statistical Society. Series B
  (Statistical Methodology)} \bibinfo{volume}{72}, \bibinfo{pages}{49--69}.
\bibitem[{Kim and Lee(2016)}]{kim_nonlinear_2016}
\bibinfo{author}{Kim, M.}, \bibinfo{author}{Lee, S.}, \bibinfo{year}{2016}.
\newblock \bibinfo{title}{Nonlinear expectile regression with application to
  {Value}-at-{Risk} and expected shortfall estimation}.
\newblock \bibinfo{journal}{Computational Statistics \& Data Analysis}
  \bibinfo{volume}{94}, \bibinfo{pages}{1--19}.
\bibitem[{Koenker and Xiao(2006)}]{Koenker2006Quantile}
\bibinfo{author}{Koenker, R.}, \bibinfo{author}{Xiao, Z.},
  \bibinfo{year}{2006}.
\newblock \bibinfo{title}{Quantile autoregression}.
\newblock \bibinfo{journal}{Publications of the American Statistical
  Association} \bibinfo{volume}{101}, \bibinfo{pages}{980--990}.
\bibitem[{Kuan et~al.(2009)Kuan, Yeh and Hsu}]{kuan_assessing_2009}
\bibinfo{author}{Kuan, C.M.}, \bibinfo{author}{Yeh, J.H.},
  \bibinfo{author}{Hsu, Y.C.}, \bibinfo{year}{2009}.
\newblock \bibinfo{title}{Assessing value at risk with {CARE}, the
  {Conditional} {Autoregressive} {Expectile} models}.
\newblock \bibinfo{journal}{Journal of Econometrics} \bibinfo{volume}{150},
  \bibinfo{pages}{261--270}.
\bibitem[{Kupiec(1995)}]{Kupiec1995Techniques}
\bibinfo{author}{Kupiec, P.}, \bibinfo{year}{1995}.
\newblock \bibinfo{title}{Techniques for verifying the accuracy of risk
  measurement models}.
\newblock \bibinfo{journal}{Social Science Electronic Publishing}
  \bibinfo{volume}{3}, \bibinfo{pages}{73--84}.
\bibitem[{Lindner(2009)}]{Lindner2009}
\bibinfo{author}{Lindner, A.M.}, \bibinfo{year}{2009}.
\newblock \bibinfo{title}{Stationarity, Mixing, Distributional Properties and
  Moments of GARCH(p, q)--Processes}. \bibinfo{publisher}{Springer Berlin
  Heidelberg}, \bibinfo{address}{Berlin, Heidelberg}.
\newblock pp. \bibinfo{pages}{43--69}.
\bibitem[{Linton and Xiao(2013)}]{linton_estimation_2013}
\bibinfo{author}{Linton, O.}, \bibinfo{author}{Xiao, Z.}, \bibinfo{year}{2013}.
\newblock \bibinfo{title}{Estimation of and inference about the expected
  shortfall for time series with infinite variance}.
\newblock \bibinfo{journal}{Econometric Theory} \bibinfo{volume}{29},
  \bibinfo{pages}{771--807}.
\bibitem[{Lucas and Klaassen(1998)}]{Lucas1998Extreme}
\bibinfo{author}{Lucas, A.}, \bibinfo{author}{Klaassen, P.},
  \bibinfo{year}{1998}.
\newblock \bibinfo{title}{Extreme returns, downside risk, and optimal asset
  allocation}.
\newblock \bibinfo{journal}{Journal of Portfolio Management}
  \bibinfo{volume}{25}, \bibinfo{pages}{71--79}.
\bibitem[{McNeil and Frey(2000)}]{mcneil_estimation_2000}
\bibinfo{author}{McNeil, A.J.}, \bibinfo{author}{Frey, R.},
  \bibinfo{year}{2000}.
\newblock \bibinfo{title}{Estimation of tail-related risk measures for
  heteroscedastic financial time series: an extreme value approach}.
\newblock \bibinfo{journal}{Journal of Empirical Finance} \bibinfo{volume}{7},
  \bibinfo{pages}{271--300}.
\bibitem[{Newey and Powell(1987)}]{newey_asymmetric_1987}
\bibinfo{author}{Newey, W.K.}, \bibinfo{author}{Powell, J.L.},
  \bibinfo{year}{1987}.
\newblock \bibinfo{title}{Asymmetric {Least} {Squares} {Estimation} and
  {Testing}}.
\newblock \bibinfo{journal}{Econometrica} \bibinfo{volume}{55},
  \bibinfo{pages}{819--847}.
\bibitem[{Nolde and Ziegel(2017)}]{nolde_elicitability_2017}
\bibinfo{author}{Nolde, N.}, \bibinfo{author}{Ziegel, J.F.},
  \bibinfo{year}{2017}.
\newblock \bibinfo{title}{Elicitability and backtesting: {Perspectives} for
  banking regulation}.
\newblock \bibinfo{journal}{The Annals of Applied Statistics}
  \bibinfo{volume}{11}, \bibinfo{pages}{1833--1874}.
\bibitem[{Owen(1990)}]{owen_empirical_1990}
\bibinfo{author}{Owen, A.}, \bibinfo{year}{1990}.
\newblock \bibinfo{title}{Empirical likelihood ratio confidence regions}.
\newblock \bibinfo{journal}{The Annals of Statistics} \bibinfo{volume}{18},
  \bibinfo{pages}{90--120}.
\bibitem[{Owen(2001)}]{owen2001empirical}
\bibinfo{author}{Owen, A.B.}, \bibinfo{year}{2001}.
\newblock \bibinfo{title}{Empirical likelihood}.
\newblock \bibinfo{publisher}{Wiley Online Library}.
\bibitem[{Peng et~al.(2015)Peng, Wang and Zheng}]{peng_empirical_2015}
\bibinfo{author}{Peng, L.}, \bibinfo{author}{Wang, X.}, \bibinfo{author}{Zheng,
  Y.}, \bibinfo{year}{2015}.
\newblock \bibinfo{title}{Empirical likelihood inference for
  {Haezendonck}-{Goovaerts} risk measure}.
\newblock \bibinfo{journal}{European Actuarial Journal} \bibinfo{volume}{5},
  \bibinfo{pages}{427--445}.
\bibitem[{Pflug(2000)}]{Pflug2000Some}
\bibinfo{author}{Pflug, G.C.}, \bibinfo{year}{2000}.
\newblock \bibinfo{title}{Some Remarks on the Value-at-Risk and the Conditional
  Value-at-Risk}.
\newblock \bibinfo{publisher}{Springer US}.
\bibitem[{Qin and Lawless(1994)}]{qin_empirical_1994}
\bibinfo{author}{Qin, J.}, \bibinfo{author}{Lawless, J.}, \bibinfo{year}{1994}.
\newblock \bibinfo{title}{Empirical {Likelihood} and {General} {Estimating}
  {Equations}}.
\newblock \bibinfo{journal}{The Annals of Statistics} \bibinfo{volume}{22},
  \bibinfo{pages}{300--325}.
\bibitem[{Rao(1962)}]{rao_relations_1962}
\bibinfo{author}{Rao, R.R.}, \bibinfo{year}{1962}.
\newblock \bibinfo{title}{Relations between {Weak} and {Uniform} {Convergence}
  of {Measures} with {Applications}}.
\newblock \bibinfo{journal}{The Annals of Mathematical Statistics}
  \bibinfo{volume}{33}, \bibinfo{pages}{659--680}.
\bibitem[{Scaillet(2004)}]{scaillet_nonparametric_2004}
\bibinfo{author}{Scaillet, O.}, \bibinfo{year}{2004}.
\newblock \bibinfo{title}{Nonparametric {Estimation} and {Sensitivity}
  {Analysis} of {Expected} {Shortfall}}.
\newblock \bibinfo{journal}{Mathematical Finance} \bibinfo{volume}{14},
  \bibinfo{pages}{115--129}.
\bibitem[{Simonato(2011)}]{simonato_performance_2011}
\bibinfo{author}{Simonato, J.G.}, \bibinfo{year}{2011}.
\newblock \bibinfo{title}{The {Performance} of {Johnson} {Distributions} for
  {Computing} {Value} at {Risk} and {Expected} {Shortfall}}.
\newblock \bibinfo{journal}{Journal of Derivatives} \bibinfo{volume}{19},
  \bibinfo{pages}{7--24}.
\bibitem[{Taylor(2008)}]{taylor_estimating_2008}
\bibinfo{author}{Taylor, J.W.}, \bibinfo{year}{2008}.
\newblock \bibinfo{title}{Estimating {Value} at {Risk} and {Expected}
  {Shortfall} {Using} {Expectiles}}.
\newblock \bibinfo{journal}{Journal of Financial Econometrics}
  \bibinfo{volume}{6}, \bibinfo{pages}{231--252}.
\bibitem[{Waltrup et~al.(2015)Waltrup, Sobotka, Kneib and
  Kauermann}]{waltrup_expectile_2015}
\bibinfo{author}{Waltrup, L.S.}, \bibinfo{author}{Sobotka, F.},
  \bibinfo{author}{Kneib, T.}, \bibinfo{author}{Kauermann, G.},
  \bibinfo{year}{2015}.
\newblock \bibinfo{title}{Expectile and quantile regression - david and
  {Goliath}?}
\newblock \bibinfo{journal}{Statistical Modelling} \bibinfo{volume}{15},
  \bibinfo{pages}{433--456}.
\bibitem[{Wang and Zhao(2016)}]{wang_conditional_2016}
\bibinfo{author}{Wang, C.S.}, \bibinfo{author}{Zhao, Z.}, \bibinfo{year}{2016}.
\newblock \bibinfo{title}{Conditional {Value}-at-{Risk}: {Semiparametric}
  estimation and inference}.
\newblock \bibinfo{journal}{Journal of Econometrics} \bibinfo{volume}{195},
  \bibinfo{pages}{86--103}.
\bibitem[{Wooldridge(1986)}]{RePEc:eee:ecochp:4-45}
\bibinfo{author}{Wooldridge, J.M.}, \bibinfo{year}{1986}.
\newblock \bibinfo{title}{{Estimation and inference for dependent processes}},
  in: \bibinfo{editor}{Engle, R.F.}, \bibinfo{editor}{McFadden, D.} (Eds.),
  \bibinfo{booktitle}{{Handbook of Econometrics}}.
  \bibinfo{publisher}{Elsevier}. volume~\bibinfo{volume}{4} of
  \textit{\bibinfo{series}{Handbook of Econometrics}}.
  chapter~\bibinfo{chapter}{45}, pp. \bibinfo{pages}{2639--2738}.
\bibitem[{Xiao and Koenker(2009)}]{xiao_conditional_2009}
\bibinfo{author}{Xiao, Z.}, \bibinfo{author}{Koenker, R.},
  \bibinfo{year}{2009}.
\newblock \bibinfo{title}{Conditional {Quantile} {Estimation} for {Generalized}
  {Autoregressive} {Conditional} {Heteroscedasticity} {Models}}.
\newblock \bibinfo{journal}{Journal of the American Statistical Association}
  \bibinfo{volume}{104}, \bibinfo{pages}{1696--1712}.
\bibitem[{Xie et~al.(2014)Xie, Zhou and Wan}]{xie_varying-coefficient_2014}
\bibinfo{author}{Xie, S.}, \bibinfo{author}{Zhou, Y.}, \bibinfo{author}{Wan,
  A.T.K.}, \bibinfo{year}{2014}.
\newblock \bibinfo{title}{A {Varying}-{Coefficient} {Expectile} {Model} for
  {Estimating} {Value} at {Risk}}.
\newblock \bibinfo{journal}{Journal of Business \& Economic Statistics}
  \bibinfo{volume}{32}, \bibinfo{pages}{576--592}.
\bibitem[{Zhu and Galbraith(2011)}]{zhu_modeling_2011}
\bibinfo{author}{Zhu, D.}, \bibinfo{author}{Galbraith, J.W.},
  \bibinfo{year}{2011}.
\newblock \bibinfo{title}{Modeling and forecasting expected shortfall with the
  generalized asymmetric {Student}-t and asymmetric exponential power
  distributions}.
\newblock \bibinfo{journal}{Journal of Empirical Finance} \bibinfo{volume}{18},
  \bibinfo{pages}{765--778}.
\bibitem[{Ziegel(2016)}]{ziegel_coherence_2013}
\bibinfo{author}{Ziegel, J.F.}, \bibinfo{year}{2016}.
\newblock \bibinfo{title}{Coherence and {Elicitability}}.
\newblock \bibinfo{journal}{Mathematical Finance} \bibinfo{volume}{26},
  \bibinfo{pages}{901--918}.

\end{thebibliography}

\end{document}